\documentclass[11pt,leqno]{amsart}
\usepackage{amsmath,amsfonts,amssymb,amscd,amsthm,amsbsy,upref}
\def\hangbox to #1 #2{\vskip3pt\hangindent #1\noindent \hbox to #1{#2}$\!\!$}

\newtheorem*{mthm}{Main Theorem}
\newtheorem*{thmA}{Theorem A}
\newtheorem*{thmB}{Theorem B}
\newtheorem*{thmC}{Theorem C}
\newtheorem{thm}{Theorem}[section]
\newtheorem{lem}[thm]{Lemma}
\newtheorem{cor}[thm]{Corollary}
\newtheorem{prop}[thm]{Proposition}
\theoremstyle{definition}
\newtheorem{ex}[thm]{Example}
\newtheorem{exs}[thm]{Examples}

\newtheorem{defin}[thm]{Definition}

\theoremstyle{remark}
\newtheorem{rem}[thm]{Remark}

\def\B{{\mathbb B}}\def\D{{\mathbb D}}

\def\E{{\mathbb E}}
\def\P{{\mathbb P}}

\def\N{{\mathbb N}}
\def\R{{\mathbb R}}

\def\cA{{\mathcal A}}
\def\cT{{\mathcal T}}

\def\cB{{\mathcal B}}

\def\cF{{\mathcal F}}

\def\cS{{\mathcal S}}
\def\cG{{\mathcal G}}

\newcommand{\tn}{|\!|\!|}

\newcommand{\keq}{\!=\!}
 \newcommand{\kneq}{\!\neq\!}
\newcommand{\kleq}{\!\leq\!}
\newcommand{\kge}{\!\ge\!}\newcommand{\kgeq}{\!\ge\!}
\newcommand{\kle}{\!<\!}
\newcommand{\kgr}{\!>\!}
\newcommand\kin{\!\in\!}
\newcommand{\ksubset}{\!\subset\!}
\newcommand{\ksupset}{\!\supset\!}
\newcommand{\ksetminus}{\!\setminus\!}
\newcommand{\kplus}{\!+\!}
\newcommand{\kminus}{\!-\!}

\newcommand{\ab}{\overline{a}}

\newcommand{\xb}{\overline{ x}}
\newcommand{\zb}{\overline{z}}
\newcommand{\jb}{\overline{\text{\em \j}}}

\newcommand{\yb}{\overline{y}}
\newcommand{\vpb}{{\overline \vp}}

\renewcommand{\l}{{\rm l}}
\newcommand{\fw}{\text{\fw}}

\newcommand{\rg}{\text{\rm rg}}

\newcommand{\supp}{{\rm supp}}

\newcommand{\spa}{{\rm span}}

\newcommand{\Sz}{\text{\rm Sz}}
\newcommand{\CB}{\text{\rm CB}}
\newcommand{\MAX}{\text{\rm MAX}}

\def\vp{\varepsilon}
\def\vpb{{\overline{\vp}}}

\newcommand{\xt}{{\tilde x}}

\newcommand{\zt}{{\tilde z}}

\newcommand{\sder}[1]{\ensuremath{^{(#1)}_S}} 
\newcommand{\si}{\ensuremath{\mathrm{I}_S}} 
\newcommand{\wi}{\ensuremath{\mathrm{I}_w}} 

\newcommand{\ie}{\textit{i.e.,}\ }

\begin{document}

\allowdisplaybreaks
\title[On Zippin's Embedding Theorem]{On Zippin's Embedding Theorem of Banach spaces into Banach spaces with bases}
\author{Th.~Schlumprecht}
\begin{abstract} We present a new proof of Zippin's Embedding Theorem, that every  separable reflexive  Banach space 
embeds into one with shrinking and boundedly complete  basis, and every Banach space with a separable dual embeds into one with a shrinking
 basis. This new  proof leads to improved versions of other embedding results.
\end{abstract}
\thanks{Research partially supported by grants  from the National Science Foundation DMS 0856148 and DMS 1160633}
\address{Department of Mathematics, Texas A\&M University, College Station, TX 77843, USA and }
\address{Faculty of Electrical Engineering, 
Czech Technical University in Prague, Zikova 4, 16627, Prague, Czech Republic}
\email{schlump@math.tamu.edu}
\keywords{Embedding into  Banach spaces with bases, Szlenk index}

\subjclass[2000]{46B03, 46B10.}
\maketitle

\tableofcontents

\section{Introduction}\label{S:1}

In 1988 M.~Zippin answered a question posed by Pe\l czy\'nski \cite[Problem I]{Pe}  in 1964 and proved the following  embedding  result.
\begin{thm}\label{T:1.1}\cite[Corollary]{Z} Every separable and reflexive Banach space embeds into  a reflexive Banach space with a basis.\end{thm}
It was shown in  \cite{DFJP}, and mentioned in \cite{Z}, that   Theorem \ref{T:1.1}  can be deduced from the  following result which answers a  question 
of Lindenstrauss and Tzafriri \cite[Problem 1.b.16]{LT}.
\begin{thm}\label{T:1.2}\cite[Theorem]{Z}  Every Banach space with a separable dual embeds into a space with shrinking basis.\end{thm}

Zippin's Theorem is the starting point of several other embedding results.  In \cite{OSZ},  it was shown that  if $X$ is a reflexive and separable
Banach space 
 and $\alpha$ is a countable ordinal  for which $\max(\Sz(X),\Sz(X^*))\le \omega^{\alpha\omega}$ then $X$ embeds into a  reflexive space $Z$ with basis
 for which $\max(\Sz(Z),\Sz(Z^*))\le \omega^{\alpha\omega}$.  Here $\Sz(Y)$ denotes the Szlenk index of a Banach space $Y$ \cite{Sz} (see Section \ref{S:4}). In \cite{FOSZ} it was shown 
 that if $X$ has a separable dual and  $\Sz(X)\le  \omega^{\alpha\omega}$, then $X$ embeds in a space $Z$  with shrinking basis for which 
 $\Sz(Z)\le  \omega^{\alpha\omega}$.  Causey \cite{Ca1,Ca2}  refined these results and proved that if  $\Sz(X)\le \omega^{\alpha }$, then 
  $X$ embeds into a space $Z$ with a shrinking basis with $\Sz(Z)\le     \omega^{\alpha+1}$,
   and   it  embeds into a space $Z$  with shrinking and boundedly complete basis for which 
   $\max(\Sz(Z)\Sz(Z^*))\le     \omega^{\alpha+1}$, in case that $X$ is reflexive and $\max(\Sz(X)\Sz(X^*))\le     \omega^{\alpha}$. Recall that by  \cite[Theorems 3.22 and
4.2]{AJO} the Szlenk index of a space with separable dual is always of the form $\omega^\alpha$, for some $\alpha<\omega_1$.
 In \cite{JZ1}  Johnson and Zheng characterized  reflexive spaces, which embed into reflexive spaces,  having an unconditional basis, and in 
 \cite{JZ2} they obtained an analogous result for spaces with separable duals. 
 
 The proof of all these embedding results start by applying Theorem \ref{T:1.1}
 or Theorem \ref{T:1.2} to embed the given space $X$ into a  reflexive space or a space with separable dual $Z$, which has a basis. Then, using the additional  properties 
 of $X$, one modifies the space $Z$  appropriately, to achieve the wanted  properties of  $Z$, without losing the embeddability of $X$ into $Z$.

The two  known proofs of Zippin's Embedding Theorem \ref{T:1.2}, namely Zippin's original  proof, as well as the  proof by  Ghoussoub, Maurey and Schachermayer \cite{GMS} start by embedding
 the given Banach space $X$ into $C(\Delta)$,  the space of continuous functions on the Cantor set $\Delta$,   and then
  passing  to subspaces and modifying  the  norm on them.  Unfortunately, neither proof provides additional information about the space with basis in which $X$ embeds.
In  this paper we will follow a different approach and  present a proof of Theorems \ref{T:1.1}  and \ref{T:1.2} which starts from  a Markushevich basis of the given space $X$, 
 and  then extends and modifies  this Markushevich basis   
  just enough to arrive to a space with shrinking basis.
 The resulting space $W$ will then be  much closer to the space $X$ and inherit several properties.
 
Our main result is as  follows. All possibly unfamiliar  notation will be introduced later.
\begin{mthm} Assume that $X$ is a Banach space with separable dual. Then $X$ embeds into a space
$W$ with a shrinking  basis   $(w_i)$  so that 
\begin{enumerate} 
\item[ a)] $\Sz(W)=\Sz(X)$,
\item[b)] if $X$ is reflexive then $W$ is reflexive and $\Sz(X^*)=\Sz(W^*)$, and 
\item[c)] if $X$ has the {\em $w^*$-Unconditional Tree Property}, then $(w_i)$ is unconditional.
\end{enumerate}
\end{mthm}
Part (a) and (b) of the Main Theorem  answer a question posed by  Pe\l czy\'nski, and sharpen the  results of 
\cite{Ca1,Ca2, FOSZ, OSZ} which were stated at the beginning of this section.
As mentioned before,  the fact that a reflexive separable  Banach space $X$, or a space  with separable dual, having the $w^*$-{\em Unconditional Tree Property } (see Section \ref{S:6}) embeds into a space $Y$ with boundedly and shrinking basis, or with shrinking basis, respectively,
 was first shown in \cite{JZ1} and \cite{JZ2}. While the proofs of the  main results of  \cite{JZ1} start out by using  
 Zippin's Theorem \ref{T:1.1} and 
 first consider an embedding of the given separable space $X$  into a reflexive space with basis, in \cite{JZ2} it was directly shown that 
  a space with separable dual and  the $w^*$-Unconditional Tree Property, embeds into  one with an unconditional and shrinking basis.
 Our argument will follow more along the lines of \cite{JZ2}  and use coordinate systems which are known to exist in every separable Banach space, namely Markushevich bases, and their multidimensional counterparts {\em Finite Dimensional Markushevish Decompositions}  (see Section \ref{S:2}).

The  Main Theorem will follow from the following two results, Theorem A and Theorem B. The  first one is  a version of the Main Theorem  for {\em Finite Dimensional Decompositions} (FDD), which will be defined in Section \ref{S:2}.
\begin{thmA} Assume that $X$ is Banach space with separable dual. Then $X$ embeds into a space
$Z$ with a shrinking    FDD  $(Z_i)$  so that 
\begin{enumerate} 
\item[ a)] $\Sz(Z)=\Sz(X)$,
\item[b)] if $X$ is reflexive then $Z$ is reflexive and $\Sz(X^*)=\Sz(Z^*)$, and 
\item[c)] if $X$ has the {\em $w^*$-Unconditional Tree Property}, then $(Z_i)$ is unconditional.
\end{enumerate}
\end{thmA}
The second result uses a construction in \cite{LT} and allows to pass from FDDs to bases.
\begin{thmB} Assume that $V$ is Banach space with an FDD $(V_j)$. Then there exists a Banach space $W$ with a basis $(w_j)$, which contains $V$ so that 
\begin{enumerate} 
\item[ a)] if $(V_i)$  is shrinking, so is $(w_j)$, and in that case   $\Sz(W)=\Sz(V)$,
\item[b)] if $V$ is reflexive so is $W$, and in this case  $\Sz(W^*)=\Sz(V^*)$, and 
\item[c)] if $(V_j)$ is an  unconditional FDD then  $(w_j)$ is an  unconditional basis.
\end{enumerate}
\end{thmB}
It it noteworthy to mention that, independently from the property of the given  space $X$, the construction  of the spaces $Z$ and $W$ is the same.
 $Z$ and $W$ inherit automatically the additional properties from $X$ mentioned in (a), (b) and (c) of the aforementioned Theorems. Since the construction is very concrete one may hope that $X$ and its superspaces $Z$ and $W$ (where $W$ is built for $V=Z$)  share also other properties.

Our paper will be organized as follows. In Section \ref{S:2} we introduce 
 {\em Finite Dimensional Markushevich  Decompositions} (FMD)  of  a separable Banach space $X$, which are obtained by {\em blocking} a given Markushevich basis.  We finish Section \ref{S:2} with a  blocking Lemma \ref{L:2.3} which shows that a given 
 FMD  can be blocked into a further  FMD  which has  the property that {\em skipped blocks} are basic sequences.
  Starting with an appropriately blocked  shrinking FMD of the  space $X$ with separable dual $X^*$  we   construct  in Section \ref{S:3} the space $Z$ with FDD $(Z_i)$, which contains $X$.  Then we  prove an FDD version of Zippin's Theorem, namely that $(Z_i)$ is shrinking (Lemma \ref{L:3.5}), and that $(Z_i)$ is boundedly complete if $X$ is reflexive (Lemma \ref{L:3.7}). Moreover we prove that if  the biorthogonal sequence  $(F_j)$ of $(E_j)$ (which is an FMD of $X^*$)
is {\em skipped  unconditional}, then $(Z_j)$ is unconditional. In the second part of Section \ref{S:3}   we construct for  a space $V$ with FDD $(V_j)$ a space
$W$ containing $V$ with a   basis $(w_j)$,  which is shrinking if $(V_j)$ is shrinking, and, moreover,  boundedly complete if $V$ is reflexive, and which is unconditional if $(V_j)$ is unconditional
 (Theorem \ref{T:3.9}). We therefore proved  Zippin's Theorems \ref{T:1.1} and \ref{T:1.2}, and, moreover,  we   reduced the  proof of Johnson's and Zheng's results
 \cite{JZ1,JZ2}
 to the problem of showing  that  the $w^*$-UTP implies the existence of FMDs which are skipped unconditional.
 Section \ref{S:4} serves as an introduction to Section \ref{S:5}. We introduce certain trees on sets and different  ordinal valued indices on them. We also introduce as an example {\em Schreier} and {\em Fine Schreier Families},  and  observe how these families can serve to measure the indices of trees.
At the end of Section \ref{S:4} we verify   some type of  {\em concentration phenomena} for families of functions defined on  maximal Schreier sets (Corollary \ref{C:4.12}).
 In Section \ref{S:5} we recall the definition of the Szlenk index  $\Sz(K)$   for bounded $K\subset X^*$  and the Szlenk index of $X$, defined by $\Sz(X)=\Sz(B_{X^*})$. We recall some, for our purposes  relevant, results from the literature. Using the  above mentioned Corollary \ref{C:4.12},  we
 prove the following result on the Szlenk index which is  of independent interest:
 
 \begin{thmC}
    If $K\ksubset B_{X^*}$ is norming $X$, then 
 $\Sz(X)\keq\min\{ \omega^\alpha: \alpha\kle\omega_1\text{ and } \omega^\alpha\kge \Sz(K)\}$.
 \end{thmC}
 
 With the help of Theorem C
 we verify the claims on the Szlenk indices in (a) and (b) of Theorems A and B at the end of Section \ref{S:5}. In our last Section \ref{S:6} we recall {\em Infinite  Asymptotic Games }
 as introduced in \cite{OS1,OS2} but with respect  to FMDs instead of FDDs and show that the main results also hold in this more general 
 framework.  We then proof the last part of Theorem A, and show that if $X$ enjoys the $w^*$-UTP a given shrinking FMD can be blocked to be skipped unconditional. 

\section{Finite Dimensional Markushevich Decompositions}\label{S:2}

In this section we  introduce {\em Finite Dimensional Markushevich  Decompositions}  of  a separable Banach space $X$. These  are the multidimensional 
versions of {\em Markushevich bases}.
 
 Let $X$ be a separable  Banach space.  
  By a result of Markushevich  \cite{Mark} (see also \cite[Theorem 1.22]{HMVZ}) $X$ admits a {\em Markushevich basis, or $M$-basis} which is  {\em $1$-norming}.
  Recall that a sequence $(e_i) \subset X$ is called
  {\em fundamental for $X$}  if  $\spa(e_i:i\in\N)$, the linear span of $(e_i)$, is  norm dense in $X$,
  and a fundamental sequence $(e_i)$ is called {\em minimal } if  $e_i\not\in\overline{\spa(e_j:j\in \N\setminus\{i\})}$, for every $i\in\N$.
  The Hahn Banach Theorem yields that the minimality of a fundamental sequence $(e_i)$ in $X$ is equivalent to  the existence of a
  unique  sequence $(f_i)\subset X^*$ which is biorthogonal to $(e_i)$.
    If $(e_i)$  is fundamental and minimal and $(f_i)$ is its biorthoganal sequence,  we say that  $(f_i)$ is {\em  total}, if for all $x\kin X$,  $f_i(x)\keq0$, for all $i\kin\N$, implies that 
    $x\keq0$. A fundamental and minimal sequence $(e_i)$, whose biorthogonals $(f_i)$ are total, is called a {\em Markushevich basis} or {\em $M$-basis}.
 If $(e_i)$ is a   Markushevich basis, the  biorthogonal sequence $(f_j)$ of $(e_i)$, is called  {\em $c$-norming } for some $c\in(0,1]$, if 
     $$ \sup_{f\in\spa(f_j:j\in\N), \|f\|\le 1} f(x)  \ge c\|x\|.$$
    
     If $(e_i)$ is a Markushevich basis and $(f_j)$ are its biorthogonals, we call a sequence
     $(E_k)$ with $E_k=\spa( e_j: n_{k-1}<j\le n_k)$,  where $0=n_0<n_1<n_2<\ldots $ are in $\N$, a {\em blocking of $(e_j)$ into finite dimensional spaces} and
     note that in that case 
          \begin{enumerate}
      \item[a)] the sequence $(E_k)$ is  {\em fundamental}, i.e. $\spa(E_k:k\in\N)$ is dense in $X$,
\item[b)] $(E_k)$ is {\em minimal}, meaning that  $E_k\cap \overline{\spa(E_j:j\in\N\setminus \{k\})}=\{0\}$, for every $k\in\N$. 
    In that case we call 
  the sequence $(F_k)$, with 
  \begin{align*}
  F_k&= \spa\big(E_j:j\in\N\setminus \{k\}\big)^\perp\\
        &=\big\{ f\in X^*: f|_{\spa(E_j:j\in\N\setminus \{k\})}=0\big\}=\spa(f_j: n_{k-1} <j\le n_k), \text{ for $k\in\N$,}\notag
         \end{align*}
  {\em the  biorthogonal sequence  to $(E_j)$}.
 \item[c)] $(F_k)$ is {\em total}, which means that for $x\in X $,  with $f(x)\keq0$, for all $f\kin F_k$ and $k\in\N$, it follows that $x\keq0$.
 \item[d)] In the  case that $(f_j)$ is $c$-norming, then    $(F_k)$ is  also $c$-norming, 
   $$\|x\|\ge c \sup_{f\in\spa(F_j:j\in\N), \|f\|\le 1} f(x).$$
       \end{enumerate}
      We call any sequence $(E_k)$ of finite dimensional subspaces of $X$ a {\em Finite Dimensional Markushevich Decomposition of $X$ (FMD) }
      if $(E_k)$ and  the sequence $(F_k)$, as defined by the first equation in (b), satisfy (a), (b) and (c).
      As we just pointed out, any blocking of an $M$-basis of $X$  is an  FMD of $X$. Conversely, it is also easy to obtain an $M$-basis from 
      an FMD. Indeed, assume that $(E_k)$ is an FMD and let $(F_k)$ be its biorthogonal sequence. First note that 
      it follows that  $F_k$ separates points of $E_k$, and  $E_k$ separates the  points of $F_k$, for each $k\in\N$, and 
      thus
       $\dim(E_k)=\dim(F_k)$, and we can find  a basis $(e^{(k)}_j:1\le j\le \dim(E_k))$      of $E_k$ and a 
       basis  $(f^{(k)}_j:1\le j\le \dim(E_k))$      of $F_k$ which is biorthogonal to 
      $(e^{(k)}_j:1\le j\le \dim(E_k))$. It follows therefore that the set  $\{ e^{(k)}_j :k\in \N, 1\le j\le \dim(E_k)\}$,   
   arbitrarily ordered into a sequence, is an $M$-basis of $X$ and $\{ f^{(k)}_j :k\in \N, 1\le j\le \dim(E_k)\}$ are the biorthogonals.

Assume that
 $(E_j)$ is an FMD of $X$.  From the minimality in (b) it follows that every $x\in \spa(E_j:j\kin\N)$ can be written uniquely as $x=\sum_{j=1}^\infty x_j$,
 with $x_j\in E_j$, for $j\kin\N$,  and $\#\{j\in\N: x_j\not=0\}<\infty$, thus we can identify $\spa(E_j:j\kin\N)$ with
 $$c_{00}( E_j)=\big\{ (x_j): x_j\in E_j,\, j\in \N, \text{ and }\#\{j\in\N: x_j\not=0\}<\infty\big\}.$$ 
 From the minimality  condition  (b)  it also  follows for all $m\in\N$, that  $X$ is the complemented sum of  $E_m$ and  the space
    $$\overline{ \spa(E_n:n\kin\N\setminus\{m\})}= ^\perp\!\!\!F_m=\{x\in X: x^*(x)=0 \text{ for all } x^*\in F_m\}.$$  Thus,  for every $m\in\N$ the
     projection $P_m^E: X\to E_m$ is bounded, where   $P_m^E(x)= x_m$,  for $x\in X$, if $x=y_m+x_m$ is the unique 
    decomposition of $x$ into  $y_m\in\overline{ \spa(E_n:n\kin\N\setminus\{m\})}$ and $x_m\in E_m$. For a finite set $A\subset \N$ we
     define $P_A^E=\sum_{m\in A} P_m^E$ and for a cofinite $A\subset \N$ we put $P_A^E=Id- \sum_{m\in \N\setminus A}P^E_m$. 
   For $x\in X$ we call the {\em support of $x$ with respect to $(E_n)$} the set 
   $$\supp_E(x)=\{ j\in\N: P_j^E(x)\not= 0\}.$$
 For $x^*\in X^*$ we   define the   {\em support of $x^*$ with respect to $(E_n)$} by 
   $$\supp_E(x^*)= \big\{ j\in\N: x^*|_{E_j}\not= 0\}.$$
    The {\em  range of $x\in X$ or $x^*\in X^*$} is the smallest interval in $\N$ containing the support of $x$, or $x^*$, and is denoted by $\rg_E(x)$, or
    $ \rg_E(x^*)$. A {\em block sequence  with respect to $(E_n)$ in $X$ or in $X^*$} is a finite or infinite  sequence $(x_n)$ in $X$,  or a sequence $(x^*_n)$ in $X^*$ for which
      $\max\rg_E(x_n)<\min\rg_{E}(x_{n+1})$ or 
    $\max\rg_E(x^*_n)<\min\rg_{E}(x^*_{n+1})$, respectively, for all $n\in\N$ for which $x_{n+1}$, or $x^*_{n+1}$ are defined.
     In the case that 
    $\max\rg_E(x_n)<\min \rg_{E}(x_{n+1})-1$ or 
    $\max \rg_E(x^*_n)<\min\rg_{E}(x^*_{n+1})-1$, respectively, we call  the  sequence   a  {\em skipped  block sequence  with respect to $(E_n)$ in $X$ or in $X^*$}.
    Note that in  finite blocks the last element does not need to have a finite range.
 
   It is easy to see that the   sequence $(F_j)$ is an FMD  of $Y=\overline{\spa(F_j:j\in\N)}$ whose biorthogonal sequence is $(E_j)$. Here we identify
   $X$ in the canonical way with a subspace of $Y^*$. Using our notation it follows then for $y\in Y$, that $\supp_F(y)=\supp_E(y)$,  and $\rg_F(y)=\rg_E(y)$.
 But since we want to apply the support and range also to elements of $X^*$ which are not 
    in $Y$, we prefer to write $\rg_E(y)$, and  $\supp_E(y)$.

   Similar to the case of $M$-bases we can of course also define blockings of an FMD $(E_j)$ as follows. If  $(E_j)$ is an FMD and $(F_j)$ is its biorthogonal sequence, then $(G_k)$ is a {\em blocking of $(E_j)$} if  $G_k=\spa(E_j: n_{k-1} < j\le n_k)$, for all $k\in\N$, and some natural numbers $0=n_0<n_1<n_2...$.
  $(G_k)$  is then also an FMD of $X$ and its biorthogonal sequence is $(H_k)$ with $H_k=\spa(F_j: n_{k-1} < j\le n_k)$, and $(H_k)$  is $c$-norming if
   $(F_j)$ was $c$-norming.

 An FMD $(E_j)$ is called  a {\em Finite Dimensional Decomposition of $X$} or  FDD, if  for every $x\in X$ there is a unique sequence $(x_j)$, $x_j\in E_j$, for $j\kin\N$, so that 
$x=\sum_{j=1}^\infty x_j$. As in the case of Schauder bases it follows    from the Uniform Boundedness Principle 
that an  FMD of  $X$ is an FDD of $X$ if and only if  the sequence $(P^E_{[1,k]}:k\in\N)$ is uniformly bounded.  
As in the case of Schauder bases we call for  an FDD $(E_n)$  the number
$b= \sup_{ m\le n} \big\| P^E_{[m,n]}\big\|$
{\em the projection constant of $(E_j)$}, and we call $(E_j)$  {\em bimonotone}  if $b=1$. We call an FDD  $(E_i)$ {\em  shrinking} if the 
biorthogonal sequence $(F_n)$  spans a dense subspace of $X^*$, and we call $(E_i)$ boundedly complete if  for every block sequence
$(x_n)$, for which $\sup_{n\in \N}\|\sum_{j=1}^n x_j\|<\infty$, the series $\sum_{j=1}^\infty x_j$ converges.
An FDD $(E_j)$ is called {\em unconditional} if  
$$c_u=
\sup\Big\{ \Big\|\sum_{j=1}^\infty \sigma_j x_j\Big\| : (\sigma_j)\in\{\pm1\}^\omega , \Big\|\sum_{j=1}^\infty  x_j\Big\|\le 1, x_j\in E_j, j\in\N\Big\}<\infty.$$
This  is equivalent with 
$$c_s =\sup \big\{\|P^E_A\|: A\subset \N, \text{ finite}\big\}<\infty,$$
and in this case $c_s\le c_u\le 2c_s$. An FDD $(E_n)$ is called $c$-unconditional if $c_u\le c$ and 
$c$-suppression unconditional if $c_s\le c$.

 We avoid to denote the biorthogonal sequence $(F_n)$ of an FMD $(E_n)$  by $(E^*_n)$ because we reserve the notion $E^*$ to the dual space  of  a space  $E$.
 Of course the   map $T_n:F_n\to E_n^*$, $x^*\mapsto x^*|_{E_n}$ is a  linear bijection,  and $\|T_n\|\le 1$, for $n\in\N$, but, unless $(F_n)$ is an FDD the inverses of  the $T_n$  may 
 not be  uniformly bounded. 
  Nevertheless, for Markushevich bases $(e_n)$ we will denote, as usual,  the biorthogonals by $(e_n^*)$. Also in case that $(E_n)$ is an FDD we will   denote the biorthogonal sequence by  $(E^*_n)$.

As in the case of bases or FDDs, we call an FMD $(E_j)$ 
 {\em  shrinking in $X$} if the span of  the biorthogonal sequence $(F_j)$ is dense in $X^*$. Note that in this case 
 $(F_j)$ is an FMD of $X^*$ whose biorthogonal sequence is $(E_j)$. 
Recall that if $X^*$ is separable then $X$ admits a  shrinking 
 $M$-basis 
  \cite[Lemma  1.21]{HMVZ}.
  
  The proof of the following observation is obtained like in the case of $M$-bases or in the case of FDDs.
  \begin{prop}\label{P:2.1}  Assume that  $(E_n) $  is an FMD of $X$.
 The following are equivalent:
\begin{enumerate}
\item  $(E_n)$  is shrinking,
\item for all $x^*\in X^*$ it follows that 
 $\lim_{n\to\infty} \| x^*|_{\spa(E_j:j> n)}\|=0,$
\item every bounded  sequence $(y_n)$, with $y_n\in\spa(E_j:j\in\N, j\ge n)$, is weakly null. 
\end{enumerate}

\end{prop}
\begin{rem} From the equivalence (1)$\iff$(3) in Proposition \ref{P:2.1} we deduce that in a reflexive space $X$ every FMD $(E_n)$ of $X$  is shrinking, and thus the biorthogonal sequence $(F_j)$ is a shrinking FMD of $X^*$. Indeed, if $y_n\in B_X\cap \spa(E_j:j\in\N, j\ge n)$, for $n\in\N$, then we can assume that $y_n$ is weakly converging to some $y\in X$, but  $y^*(y)$  must vanish for all $y^*\in \spa (F_j:j\kin \N)$ and it follows therefore  that $y=0$.
 
\end{rem}

\begin{proof}[Proof of Proposition \ref{P:2.1}] Let $(F_n)$ be the biorthogonal sequence of $(E_n)$.\\
``(1) $\Rightarrow$ (2)'' If   $(E_n)$ is shrinking and $x^*\in X^*$  we can find  for an arbitrary $\vp>0$ an element $y^*=\sum   f_j\in\spa(F_j:j\in\N)$, $f_j\in F_j$, for $j\kin\N$, so that
 $\|x^*-y^*\|<\vp$. Thus 
  $$\limsup_{n\to\infty} \| x^*|_{\spa(E_j:j> n)}\|\le \vp +  \lim_{n\to\infty} \| y^*|_{\spa(E_j:j> n)}\|  =\vp,$$
  which proves our claim since $\vp>0$ is arbitrary.
  
  \noindent
  ``(2) $\Rightarrow$ (1)''  Assume (2) is satisfied and let $x^*\kin X$ and $\vp\kgr0$. For large enough $m$ it follows that $\| x^*|_{\spa(E_j:j> m)}\|\kle\vp$.
   Now let $y^*\kin X^*$ be a Hahn-Banach extension  of  $x^*|_{\spa(E_j:j> m)}$. Then
   $x^*-y^*\in \spa(F_j:j\ge m)$ (since $(x^*-y^*)|_{\spa(E_j:j <m)}\equiv 0$) and $\| x^*- (x^*-y^*)\|=\|y^*\|\kle\vp$. 
   
   \noindent
   ``$\neg$(3)$\Rightarrow \neg$(2)''  Assume $x_n\kin B_{X}\cap \spa(E_j:j\kgr n)$ and $(x_n)$ is not weakly null. After passing to a subsequence we can assume that there is an  $x^*\kin B_{X^*}$ with 
   $|x^*(x_n)|\ge \vp>0$.
   Then $$\|x^*|_{\spa(E_j:j> n)}\|\kge |x^*(x_n)|\ge \vp,$$
   thus (2)  is not satisfied.
   
      \noindent
   ``$\neg$(2)$\Rightarrow \neg$(3)'' Assume that $\|x^*|_{\spa(E_j:j> n)}\|\ge\vp>0$ for all $n$.
   Then choose $x_n\in B_{X}\cap \spa(E_j:j> n)$ with  $|x^*(x_n)|\ge \vp/2$. Thus (3) is not satisfied.    
\end{proof}

We finish this introductory section with the following easy, and in similar versions well known observation, which will be crucial  for future arguments.
\begin{lem}\label{L:2.2} Let  $X$  be a separable Banach space. Assume that   $(E'_j)$  is a $1$-norming  FMD of $X$ and $(F'_j)$ is its biorthogonal 
 sequence.  
 Then $(E'_j)$ can be blocked to an FMD $(E_n)$ satisfying  with its biorthogonal sequence $(F_n)$  the following conditions  for every $m\kleq n$ in $\N$.
     \begin{align}
   \label{E:2.2.1}&\text{For all }e^*\kin ( E_m\!+\!E_{m+1}\!+ \ldots +E_n)^*\text{ there exists }x^*\kin F_{m-1}\!+\!F_m\!+\!\ldots +F_{n+1}\text{ so that}\\
     &\qquad x^*|_{E_m\!+\!E_{m+1}\!+ \ldots + E_n}= e^* \text{ and } \|x^*\|\le2.5\|e^*\|, \notag \\
      \label{E:2.2.2}&\text{for all } f^*\kin( F_m\!+\!F_{m+1}\!+ \ldots +F_{n})^* \text{ there exists } z\kin E_{m-1}\!+\!E_m\!+\!\ldots +E_{n+1}\text{ so that}\\
     & \qquad z|_{F_m\!+\!F_{m+1}\!+ \ldots + F_{n}}= f^* \text{ and } \|z\|\le \ 2.5 \|f^*\|,\notag \\ 
     \label{E:2.2.3}&\text{for all  }x^*\kin F_m\!+\!F_{m+1} \!+ \ldots+ F_{n} \\
   &\qquad \|x^*\|\le 2.5      \big\| x^*|_{E_{m-1}+E_m+ \ldots +E_n+E_{n+1}}\big\|=2.5 \sup_{x\in E_{m-1}+E_m+ \ldots +E_{n+1}, \|x\|\le 1} |x^*(x)|, \notag \\
      \label{E:2.2.4} &\text{for all   $x\in  E_m\!+\!E_{m+1}\!+ \ldots +E_{n}$} \\ 
  &\qquad\|x\|\le  2.5  \big\|x|_{F_{m-1}+F_{m}+\ldots  + F_n+F_{n+1}}\big\|= 2.5 \sup_{x^*\in F_{m-1}+F_{m}+\ldots  + F_n+F_{n+1}, \|x^*\|\le 1} |x^*(x) |.\notag 
         \end{align}
         Here  we let $E_0$ and $F_0$ be the null spaces in $X$ and $X^*$, respectively.

 \end{lem}
The proof will follow  using repeatedly the following Lemma.

   \begin{lem}\label{L:2.3} Let $X$ be a Banach space and let
    $Y'$ be a (not necessarily closed) subspace  of $X^*$ for which $B_{Y'} $ is $w^*$-dense in $B_{X^*}$.
 Assume that  $E\subset X$  is  finite dimensional, and $\vp>0$ .
 Then there is a finite dimensional subspace $F\subset Y'$, so that every
  $e^*\in E^*$ can be extended to an element $x^*\in F$, with $\|x^*\|\le (1+\vp)\|e^*\|$. 
 \end{lem}

\begin{proof} Let $\delta\in(0,\frac12)$, and choose a $\delta$-net $(e^*_j)_{j=1}^N$ in $S_{E^*}$, and let $x^*_j\in S_{X^*}$ be a Hahn-Banach extension of $e^*_j$, for $j=1,2,\ldots ,N$.
 By the assumption that $B_{Y'}$ is $w^*$-dense in $B_{X^*}$ we can choose $(y^*_j)_{j=1}^N\subset B_{Y'}$, so  that $\|y^*_j|_E-e^*_j\|<\delta$ for all $j=1,2,\ldots, N$. Let $F=\spa(y^*_j:j=1,2,\ldots, N)$ and consider the 
   restriction map $T: F\to E^*$, $x^*\mapsto x^*|_E$. By our construction $T(B_F)$ is $2\delta$- dense in $B_{E^*}$. 
   
   Now let $e^*\in B_{E^*}$. We can successively choose $x^*_1, x^*_2, x^*_3, \ldots $
   so that   $\|x^*_n\|\le (2\delta)^{n-1}$ and  
\begin{align*}\big\| e^*\kminus \big(T(x^*_1)\kplus T(x^*_2)+ \ldots T(x^*_{n-1})\big)-T(x^*_n)\big\|&\le  2\delta\big\|(e^*-\big(T(x^*_1)+T(x^*_2)+ \ldots T(x^*_{n-1})\big)\big\|\\
   &\le (2\delta)^n,
   \end{align*}
and thus, letting  $x^*=\sum_{n=1}^\infty x^*_n\in F$, we deduce that $e^*=T(x^*)$ and $\|x^*\|\le \frac1{1-2\delta}$. Choosing $\delta>0$ sufficiently small, we obtain 
our claim.
\end{proof}

     \begin{proof}[Proof of Lemma \ref{L:2.2}]      Define $X'=\spa(E'_j:j\kin\N)$ and $Y'=\spa(F'_j:j\kin\N)$.  Since $(E'_j)$ is a $1$-norming FMD, 
   $B_{Y'}$ is $w^*$-dense  in $B_{X^*}$, and moreover 
      the map $T:X\to Y^*$,  defined by
     $T(x)(y)=y(x)$, for $x\kin X$ and $y\kin Y$,  is an isometric embedding. It follows that $B_X$ and, therefore also $B_{X'}$, is $w^*$-dense  
     in $B_{Y^*}$.  Let $\rho\kgr1$ with $\rho+\rho^2<2.5$. 
   Inductively, we choose $0=n_0<n_1<n_2<\ldots$ in $\N$, so that  for all $k\in\N_0$  the following two conditions  hold.
         \begin{align}
      \label{E:2.2.6} &\text{For all $e^*\in (E'_1+E'_2+\ldots+ E'_{n_k})^*$ there is an $x^*\in F'_1+F'_2+\ldots+F'_{n_{k+1}}$ so that}\\
   &\qquad\qquad x^*|_{E'_1+E'_2+\ldots+ E'_{n_k}}=e^*\text{ and } \|x^*\|\le \rho \|e^*\|,\notag\\
   \label{E:2.2.7} &\text{for all $f^*\in (F'_1+F'_2+\ldots +F'_{n_{k}})^*$ there is an  $ x\in E'_1+E'_2+\ldots +E'_{n_{k+1}}$ so that}\\
   &\qquad\qquad x|_{F'_1+F'_2+\ldots +F'_{n_k}}=f^*\text{ and } \|x\|\le \rho \|f^*\|\notag 
   \end{align}
(with $E'_1+E'_2+\ldots + E'_0=\{0\}$ and $F'_1+F'_2+\ldots +F'_0=\{0\}$).

   For $k=0$ we choose $n_1=1$ and note that \eqref{E:2.2.6} and \eqref{E:2.2.7} are trivially satisfied.

   Assume that we have chosen  $n_k$ for some $k\ge 1$.
  We  first apply Lemma \ref{L:2.3} to $E=E'_1\!+\!E'_2\!+\!\ldots+ E'_{n_k}$  and $Y'$ to obtain
   a finite dimensional subspace $F'\subset Y'$ so that every $e^*\in ( E'_1\!+\!E'_2\!+\!\ldots+ E'_{n_k})^*$ can be extended to an element $x^*\in F'$, with $\|x^*\|\le \rho \|e^*\|$.
   Then  we apply  Lemma \ref{L:2.3} to the Banach space $Y=\overline{Y'}$, instead of $X$,  to $X'$ (recall that $B_{X'}$ is $w^*$ dense in $B_{Y^*}$),  and  to  $F=F'_1+F'_2+\ldots+ F'_{n_k}$  to obtain
   a finite dimensional subspace $E'$ of $X'$ so that every $f^*\in(F'_1+F'_2+\ldots+ F'_{n_k})^*$ can be extended to an element $x\in E'$  with $\|x\|\le \rho \|f^*\|$.
  
   Because $E'$ and $F'$ are finite dimensional subspaces of $X'$ and $Y'$, respectively, there is some $n_{k+1}>n_k$ in $\N$ so that
   $E'\subset E'_1+E'_2+\ldots	+ E'_{n_{k+1} }$ and
    $F'\subset F'_1+F'_2+\ldots+ F'_{n_{k+1} }$. This finishes the recursive definition of $n_k$, $k\in\N$.  
    
    We define $E_k=E'_{n_{k-1}+1}+E'_{n_{k-1}+2}+\ldots +E'_{n_k}$ and $F_k=F'_{n_{k-1}+1}+F'_{n_{k-1}+2}+\ldots+ F'_{n_k}$.
    
    In order to verify \eqref{E:2.2.1} let $m\le n$ and $e^*\in (E_m+E_{m+1}+\ldots+ E_n)^*$.  We first apply \eqref{E:2.2.6} to obtain
    $z^*\in F_1+F_2+\ldots +F_{n+1}$ which is an extension of $e^*$ with   $\|z^*\|\le \rho \|e^*\|$. If $m\le 2$ we can choose $x^*=z^*$. Otherwise
     we apply again \eqref{E:2.2.6} to $z^*|_{E_1+E_2+ \ldots +E_{m-2}}$ and extend it to an element $y^*$ in $F_1+F_2+\ldots+ F_{m-1}$ with
     $\|y^*\|\le \rho\|z^*|_{E_1+E_2+ \ldots+ E_{m-2}}\|\le \rho^2 \|e^*\|$ and finally put $x^*=z^*-y^*$.
     Since $x^*$ vanishes on all $E_j$ with $j\in\N\setminus[m-1,n+1]$,  it follows that  $x^*\in F_{m-1}+F_m+ \ldots+ F_n+F_{n+1}$. It is also clear that $x^*$ extends $e^*$ and 
     since $\|x^*\|\le \|z^*\|+\|y^*\| \le \rho \|e^*\|+ \rho^2 \|e^*\|$, we deduce \eqref{E:2.2.1}.
     The verification of \eqref{E:2.2.2} can be accomplished similarly to the  proof  of  \eqref{E:2.2.1}.

     To show \eqref{E:2.2.3} let $x^*\in F_m+F_{m+1}+\ldots +F_n$ and let $\eta>0$.  We can choose  $x\in S_X$ so that 
     $|x^*(x)|\ge \|x^*\|-\eta$. We view $x$ as an element of $Y^*$  and put $$f^*=x|_{F_m+F_{m+1}+\ldots +F_n}\in 
      (F_m+F_{m+1}+\ldots +F_n)^*.$$ Using \eqref{E:2.2.2} we can extend $f^*$ to an element $z\in E_{m-1}+E_{m} +\ldots+ E_n+ E_{n+1}$, with $\|z\|\le 2.5$ and 
      thus 
    $$\|x^*\|-\eta\le |x^*(x)|=|x^*(z)| \le(2.5)\sup_{y\in E_{m-1}+E_m+\ldots+ E_{n}+E_{n+1}, \|y\|\le 1} |x^*(y)|,$$
    which implies our claim since $\eta>0$ was arbitrary.
    
    \eqref{E:2.2.4} can be shown the same way using \eqref{E:2.2.1}.
      \end{proof}

     From \eqref{E:2.2.3} we easily observe the following
\begin{cor}\label{C:2.4} Assume  the sequence $(E_n)$ is an FMD of $X$  with biorthogonal sequence $(F_j)$ satisfying the conclusions of Lemma \ref{L:2.2}.
Then every  skipped block in $X$  and  block sequence $(y_j)$ in $Y=\overline{\spa(F_j:j\in\N)}$ 
 is basic with a projection constant  not larger than $2.5$.
\end{cor}

\section{Construction of $Z$ and $W$ and proof of Zippin's Theorems}\label{S:3}

Throughout this section  $X$ is a Banach space whose dual $X^*$ is separable.
We  also assume that we have chosen a shrinking FMD  $(E_n)$  of $X$ which, together with its biorthogonal sequence $(F_n)$, satisfies the 
conclusions  of Lemma \ref{L:2.2}.
 The following observation was made in \cite{Jo2} in the FDD case. The proof in the FMD case is the same. 
\begin{lem}\label{L:3.1} Let $(\vp_k)\subset (0,1]$  be given. Then there is an increasing sequence $(n_k)\subset \N$, so that for each $x^*\kin B_{X^*}$  and each $k\in\N$ there is a $j_k\in [n_{k}, n_{k+1}]$, so that $\|x^*|_{E_{j_k}}\|\le \vp_k$.  
\end{lem}
\begin{proof} Our conclusion follows from iterating the following claim.

\noindent{\bf Claim:} for any $\delta>0$ and any $m\in\N$ there is an $n>m$ in $\N$ so that for each $x^*\in B_{X^*}$ there is a $j\in[m,n]$ with $\|x^*|_{E_j}\|\le \delta$.

Assume that our claim is not true, and for each $n\ge m$ we could choose $x^*_n\in B_{X^*}$ so that $\|x^*_n|_{E_j}\|\ge \delta$ for all $j\in[m,n]$.
 After passing to a subsequence we can assume that $(x^*_n)$ $w^*$- converges to some $x^*\in B_{X^*}$. But then it follows that $\|x^*|_{E_j}\|\ge \delta$, for all 
 $j\ge m$ which contradicts property (2) in Proposition \ref{P:2.1}.
\end{proof}
We choose a decreasing sequence $(\vp_k)\subset (0,1)$ with $\sum_{k\in \N} \vp_k<\frac1{50} $ and let $(n_k)\subset\N$ satisfy the conclusion of  Lemma \ref{L:3.1}. Note that this choice implies that  $n_{k+1}>n_k+2$, for $k\kin\N$.
%
Then we define
\begin{equation} \label{E:3.1}  
D^*=\big\{ x^*\in {X^*} :\forall k\kin\N\, \exists j\kin [n_k, n_{k+1}]\quad x^*|_{E_j}\equiv 0\big\},\, \text{and }B^*=D^*\cap B_{X^*}.
\end{equation}

\begin{lem}\label{L:3.2}  $B^*$ is  $\frac1{10}$-dense in $B_{X^*}$.  \end{lem}
\begin{proof} Let $x^*\in B_X^*$, and choose, according to Lemma \ref{L:3.1},
$j_k\in [n_{k}, n_{k+1}]$, for each $k\in\N$, so that $\|x^*|_{E_{j_k}}\|\le \vp_{k}$. 
 In the case that for some $k\in\N$ we have  $j_{k+1}=j_k+1$ we change $j_k$ and $j_{k+1}$ in the following way:
 First note that  $j_{k+1}=j_k+1$ only happens if $j_k=n_{k+1}-1 $ and $j_{k+1}=n_{k+1}$ or
  $j_k=n_{k+1} $ and $j_{k+1}=n_{k+1}+1$. In that case we redefine $j_k=j_{k+1}=n_{k+1}$.
  Then we define $K=\{k\in \N: j_{k}\not= j_{k-1}\}$. This, and the  above observed fact, that $n_{k+1}>n_k+2$, for $k\kin\N$,  implies that $(j_k:k\in K)$ is a skipped sequence in $K$, and we still have 
  $j_k\in [n_{k}, n_{k+1}]$ and $\|x^*|_{E_{j_k}}\|\le \vp_{k}$,
for each $k\in K$.

 Applying for each $k\in K$,
part \eqref{E:2.2.1} of Lemma \ref{L:2.2} to $e^*_k=x^*|_{E_{j_k}}$,
 we obtain an extension of $e^*_k$ to  $f_k\in F_{j_k-1}+F_{j_k}+F_{j_k+1}$ with $\|f_k\|\kleq2.5 \vp_k$.
Then we define  
$$y^*=x^*-\sum_{k\in K} f_k \text{ and } z^*=\frac{20}{21}{y^*}.$$
Since $\|y^*\|\le \|x^*\| + \frac1{20}\le \frac{21}{20}$  it follows that $\|z^*\|\le 1$ and  thus  $z^*\in B^*$ and 
$$\|x^*-z^*\| = \|x^*-y^*\big\| +\Big\| y^* -\frac{20}{21}{y^*}\Big\|\le \|x^*-y^*\|+\frac{1}{20} <\frac1{10},$$
which finishes the proof of our claim. \end{proof}
From now on  we assume, possibly after renorming $X$,  that
\begin{equation}\label{E:3.2}
\|x\|=\sup_{x^*\in B^*} |x^*(x)| \text{ for all $x\in X$,}
\end{equation} 
in other words we assume that $B^*\subset B_{X^*}$, as defined in \eqref{E:3.1}, is $1$-norming the space  $X$.

For $x^*\in D^*$ we let $\jb=(j_k)\in \prod_{k=1}^\infty [n_{k}, n_{k+1}]$ so that $x^*|_{E_{j_k}}\equiv 0$ and put
$x^*_k= P^{F}_{(j_{k-1}, j_k)}(x^*)$ for $k\kin\N$ (with $j_0=0$). Since $(E_j)$ is shrinking it follows from Proposition \ref{P:2.1} and Lemma \ref{L:2.2}  for $m\in\N$  that 
\begin{align*}
\Big\|x^*-\sum_{k=1}^m x^*_k\Big\|&\le 3 \sup_{x\in \spa(E_j:j\ge j_{m}), \|x\|\le 1} \Big|\Big(x^*\!-\!\sum_{k=1}^m x^*_k\Big)\!(x)\Big|
= 3\big\| x^*\vert_{\spa(E_j:j\ge j_{m})}\big\|\to_{m\to\infty} 0.
\end{align*}
Thus 
\begin{equation}\label{E:3.4} x^*=\sum_{k=1}^\infty x^*_k  \text{ and this series converges in norm, for all $x^*\in D^*$.}
\end{equation}
Lemma \ref{L:2.2}  also yields  for $m\le n$ that 
\begin{align}\label{E:3.3}
\Big\|\sum_{k=m}^n x^*_k\Big\|&\le 3\sup_{x\in E_{j_{m-1}}+\ldots+ E_{j_n}, \|x\|\le 1} \sum_{k=m}^n x^*_k(x)\\
 &=3\sup_{x\in E_{j_{m-1}}+\ldots+ E_{j_n}, \|x\|\le 1}\sum_{k=1}^\infty  x^*_k(x)\le 3\|x^*\|.\notag
\end{align} 
We define
\begin{align}\label{E:3.4a} \D^*&=\left\{ (x^*_k)\subset X^*:  \begin{matrix}\exists (j_k)\kin\prod_{k=1}^\infty[n_k,n_{k+1}] \text{ so that }\\
     \text{$\rg_E(x^*_k)\subset (j_{k-1},j_{k})$, for  $k\in\N$,  $\big\|\sum_{k=1}^\infty x^*_k\big\|<\infty$} \end{matrix}\right\}\\
     \intertext{(In the definition of $\D^*$  it is possible that $j_{k-1}=j_{k}= n_{k}$ or $j_k=j_{k-1}+1$, and that in either case $x^*_k\equiv 0$) and }
\label{E:3.4b}\B^*&= \D^*\cap \Big\{ (x^*_k)\subset X^*: \Big\|\sum_{k=1}^\infty x^*_k\Big\|\le 1\Big\}. \end{align}
We can rewrite the sets $D^*$ and $B^*$ as 
\begin{align}\label{E:3.5}D^*&=\Big\{ \sum_{k=1}^\infty x^*_k: (x^*_k)\in \D^* \Big\} \text{ and }
                   B^*=D^*\cap B_{X^*}=\Big\{ \sum_{k=1}^\infty x^*_k: (x^*_k)\in \B^*\Big\}.
\end{align}

We now construct the space $Z$ with  shrinking FDD $(Z_j)$  which contains $X$.
 The important point  in the construction of the space will be that $\B^*$  will become the $1$-norming set of $Z$, and that the similarities between
 $X$ and $Z$ stem from  the similarities of the sets $D^*$ and $\D^*$, and $B^*$ and $\B^*$, respectively.
We put   for $k\in\N$
\begin{equation}\label{E:3.7} 
Z_k=\spa(E_j: n_{k-1}<j<n_{k+1}), \text{ for $k\in\N$ (as before, $n_0=0$),}
\end{equation}
and note that for $(x^*_k)\in\B^*$ and each  $k\in\N$ it follows that  $\rg_E(x^*_k)\subset (n_{k-1},n_{k+1})$, $x^*_k$ can therefore be seen as functional acting on $Z_k$.
For $z=(z_k)\in c_{00}(Z_k)$  we define
\begin{equation}\label{E:3.8} \|z\| =\sup\Bigg\{ \sum_{k=1}^\infty x^*_k(z_k):  (x^*_k)\in \B^*\Bigg\}.
\end{equation}
We define $Z$ to be the completion of $c_{00}( Z_k)$ with respect to $\|\cdot\|$. We will from now on consider the elements of $\D^*$   to be elements of $Z^*$.
If $x^*\in X^*$ with $\rg_E(x^*)\subset (n_{k-1},n_{k+1})$, for some $k\in\N$,
 we can identify $x^*$ with the sequence $(x^*_m)\in \D^*$, with $x^*_m=x^*$, if $m=k$, and $x^*_m=0$, otherwise.
Thus 
  we can consider $x^*$ to be an element of $Z^*$. 

The following Proposition gathers some properties  of the space $Z$, and shows how $Z$ inherits  the properties of $X$.

\begin{prop}\label{P:3.3} (Properties of the space $Z$)
\begin{enumerate}
\item The map
$$I:c_{00}( E_j)\to Z, \quad x\mapsto \big(P^E_{(n_{k-1},n_{k+1})}(x): k\in\N\big)   $$
extends to an isometric embedding from $X$ into $Z$.

\item
 $(Z_j)$ is an FDD of $Z$ whose projection constant is not larger than $3$.

\item
For a sequence  $\jb=(j_k)\in \prod_{k=1}^\infty [n_{k}, n_{k+1}]$,  define
$$U_{\jb}=\big\{x^*\in X^* : \forall k\in\N, x^*|_{E_{j_k}}\equiv 0\big\} .$$  Then 
$U_{\jb}$ is a $w^*$-closed subspace of $X^*$ and the map 
$$\Phi_{\jb}:U_{\jb} \to Z^*, \quad x^*\mapsto    \big(P^F_{(j_{k-1},j_{k})}(x^*):k\in\N\big) ,$$
is an isometric embedding which is continuous with respect to the $w^*$-topology of $X^*$ restricted to $U_{\jb}$ and the $w^*$-topology of $Z^*$.

\item $\B^*$ is  a $w^*$-compact subset of $B_{Z^*}$ which is $1$-norming $Z$ and 
  the restriction of $I^*: Z^*\to X^*$ to the set   $\D^*$ is a norm preserving map  from
 $\D^*$  onto $D^*$.
  
Moreover, if $(z^*_i)$ is a skipped block with respect to the FDD $(Z^*_k)$ whose elements are in $\D^*$, then $\big(I^*(z^*_i)\big)$ is 
a skipped  block in $D^*$ with respect to $(F_j)$ which is isometrically equivalent to $(z^*_i)$.

 \item  Let $Y$ be a Banach space   and  let $T_k:Y\to Z_k$, be linear, for $k\in\N $, and assume that 
  $T_k\equiv 0$ for all but finitely many $k$. 
 Define $$T : Y\to  Z, \quad y\mapsto (T_k(y):k\in\N).$$
For  every $\xb^*=(x^*_k)_{k=1}^\infty\in \B^*$  define
 $$T^*_{\xb^*}= T^*\big|_{\overline{\spa(x^*_k:k\in \N)}} : \overline{\spa(x^*_k:k\in \N)}\to Y^*,    \quad \sum a_k x^*_k \mapsto   \sum  a_k x^*_k\circ T_k.$$
Then 
 \begin{equation} \label{E:3.3.1} \|T\|_{L(Y,Z)}=\sup_{\xb\in \B^*}  \big\| T^*_{\xb^*}\|.\end{equation}
We expressed therefore the norm of an operator $T:X\to Z$ by the norm of its adjoint restricted to spaces of the form 
$\overline{\spa(x^*_j:j\in \N)}$ with $(x^*_j)\in \B^*$.
 \end{enumerate}
 \end{prop}
 \begin{proof} To verify  (1) let $x\in c_{00}( E_j)$ and note that 
\begin{align*}
\| I(x)\| =\sup_{(x^*_k)\in \B^*}   \Big|\sum_{k=1}^\infty  x^*_k\big(P^E_{(n_{k-1},n_{k+1})}(x) \big)\Big| 
 =\sup_{(x^*_k)\in \B^*}  \Big|\sum_{k=1}^\infty  x^*_k(x) \Big| 
 =\sup_{x^*\in B^*}| x^*(x)|=\|x\|.
\end{align*}

\medskip\noindent (2)   Let $m\le n$ and $z=(z_k)\in c_{00}(Z_k)$. Then it  follows from \eqref{E:3.3} 

$$\big\|P^Z_{[m,n]}(z)\big\|=\sup _{(x^*_k)\in \B^*}  \Big\|\sum_{k=m}^n x^*_k z_k\Big\|\le\sup _{(x^*_k)\in \B^*} \Big\|\sum_{k=m}^n x^*_k\Big\|\cdot  \| z\|\le 3\|z\|.$$

\medskip\noindent (3) For $\jb=(j_k)\in \prod_{k=1}^\infty [n_k, n_{k+1}]$ the space
$U_{\jb}=\{x^*\kin X^* : \forall k\kin\N, x^*|_{E_{j_k}}\equiv 0\} $ is    clearly a $w^*$-closed   subspace of $X^*$.
If $x^*\in U_{\jb}$, with $\|x^*\|=1$ put $x^*_k=P^F_{(j_{k-1}, j_k )}(x^*)$, $k\in\N$, and $z^*=(x^*_k)$. On the one hand it follows
 from the definition of the norm on $Z$ that  $z^*\in \B^*$ and, thus, $\|z^*\|\le 1$. On the other hand it follows 
for all $x\kin X$ that $z^*(I(x))= \sum_{k=1}^\infty x^*_k(x)=x^*(x)$ and thus, since $I$ is an isometric embedding, $\|z^*\|\kgeq \|x^*\|\keq1$. 
Thus $\Phi_{\jb}$  is an isometric embedding from $U_{\jb}$ into $Z^*$. 

In order to show that $\Phi_{\jb}$ is $w^*$-continuous, it is enough to show that $\Phi_{\jb}$  restricted to  the unit ball is $w^*$ continuous.
 Then the $w^*$-continuity on all of $U_{\jb}$
  follows from the already observed fact that $U_{\jb}$ is $w^*$-closed, and the Theorem of Krein-Smulian, which says that a subspace a dual space 
 is $w^*$ closed if its intersection with the unit ball is $w^*$-closed (c.f. \cite[V.5 Corollary 8]{DS})

In order to  show that $\Phi_{\jb}$  is  $w^*$-continuous on $B_{U_{\jb}}$,  let $(x^*(n))$ be a sequence in $B_{U_{\jb}}\subset B_{X^*}$ which $w^*$-converges to $x^*$, and let 
$x^*_k=P^F_{(j_{k-1},j_k)}(x^*)$ and 
$x^*_k(n)=P^F_{(j_{k-1},j_k)}(x^*(n))$, for  $k,n\kin\N$. 
 It follows that for each $k\kin\N$ the sequence $x^*_k(n)$    converges to  $x^*_k$ in $Z^*_k$ and thus  $(x^*_k(n))$  converge in $w^*$ to $x^*_k$ as functionals on $Z^*$, which act on the $k$-th coordinate. 
 But this implies that the  sequence $\Psi_{\jb}\big(x^*(n)\big)$ converges point wise  on a dense set of $Z$ to $\Phi_{\jb}(x^*)$. Since $(x^*(n):n\kin\N)$ is bounded
 we deduce that   $\big(\Psi_{\jb}(x^*(n)):n\kin\N\big)$ $w^*$-converges  in $Z^*$ to $\Phi_{\jb}(x^*)=(x^*_k)\in \B^*\subset B_{Z^*}$.
 
\medskip\noindent (4)
In order to show that $\B^*$ is $w^*$-compact let $z^*(n)= (x^*_k(n))\in \B^*$, for $n\kin\N$, and assume that $(z^*(n))$ converges in $w^*$ to some $z^*\in Z^*$.
After  passing to subsequences we can assume that there is a sequence $\jb=(j_k)\in \prod_{k=1}^\infty[n_{k}, n_{k+1}] $, so that 
$\rg_E(x^*_k(n))\subset (j_{k-1}, j_k)$ for all $k\in\N$ and all $n\ge k$. This implies that  the sequence $(x^*(n))$, with  $x^*(n)=\sum_{k\in\N}^\infty x^*_k(n)$, for $n\in\N$, converges  $w^*$ to some element 
$x^*$ which is in $U_{\jb}\cap B_{X^*}$.  It follows therefore that $z^*=\Phi_{\jb}(x^*)$, and, thus,  that $z^*\in \B^*$.

If $z^*=(x^*_k)\in \D^*$ then there is some   $\jb=(j_k)\in \prod_{k=1}^\infty [n_{k},n_{k+1}]$,  so that  $x^*=\sum x^*_k\in U_{\jb}$  and 
$z^*=\Phi_{\jb}(x^*)$.
By part (3) it follows that $\|z^*\|=\|x^*\|$, and since for all $x\in X$
$$z^*(I(x))=\sum x^*_k\big(P_{(n_{k-1},n_{k+1})}^E( x)\big)=\sum  x^*_k (x)= x^*(x),$$
  it follows that $x^*=I^*(z^*)$. 
Note also that $D^*$ consists of the union of all spaces $U_{\jb}$,
with $\jb\in\prod_{k=1}^\infty[n_{k}, n_{k+1}]$, and that  $I^*\circ\Phi_{\jb}(x^*)=x^*$ for all $x^*\kin U_{\jb}$.  It follows therefore that the image of $I^*|_{\D^*}$ is all of $D^*$.

 If  $(z^*_i)\subset \D^*$ is a skipped block with respect to $(Z^*_k)$, we choose $(m_i)\subset \N$ increasing so that 
 $\rg_{Z^*} (z^*_i)\subset (m_{i-1},m_i)$,  for $i\in\N$ (with $m_0=0$). We write  for $i\kin\N$  the element $z^*_i$ as $z^*_i= (x^*_{k}: k\keq m_{i-1}\kplus1, m_{i-1}\kplus2,\ldots, m_{i}-1)\subset X^*$, with 
  $\rg_{E} (x^*_{k})\subset (j_{k-1}, j_{k})$,  for $ k= m_{i-1}+1, m_{i-1}+2,\ldots m_{i}-1$,  where 
   $j_ {k}\in [n_k,n_{k+1}]$, for   $k=m_{i-1}, m_{i-1}+1,\ldots, m_{i}-1$.
   Finally we put $x^*_{m_i}=0$ for $i\in\N$, and deduce that  $(x^*_k)\subset U_{\jb}$ with  $\jb=(j_k)$.
    It follows that all the $z^*_i$, together with their linear combinations 
 are in the image of $U_{\jb}$ under $\Phi_{\jb}$ and it follows from the already verified facts, that $(z^*_i)$ is isometrically equivalent to its inverse image 
 which is the sequence  $(I^*(z^*_i))$.

 \medskip\noindent(5) We  deduce from the definition \eqref{E:3.8} of the norm on $Z$ that  
  \begin{align*}\|T\|=\sup_{y\in B_Y} \|T(y)\|&=\sup_{y\in B_Y}\sup_{(x^*_k)\in \B^*} \Big|\sum  x^*_k\big(T_k(y)\big)\Big|\\
  &= \sup_{\xb^*=(x^*_j)\in \B^*} \sup_{y\in B_Y}\Big|T_{\xb^*}\Big(\sum  x ^*_k\Big)(y)\Big|
  =\sup_{\xb^*\in\B^*}\big\|T_{\xb^*}\big\| 
 \end{align*}
  which proves our claim (7).
 \end{proof} 
\begin{rem}\label{R:3.4} Note  that  $I^*|_{\D^*}$ is norm preserving but   not  injective. Indeed,  let $x^*\kin B^*$ have the property that  for some  $k_0\in\N$ there are $j,j'\in [n_{k_0}, n_{k_0+1}]$, with  $j<j'-1$,
 $x^*|_{E_j}= x^*|_{E_{j'}}=0$ , and  there is $i\in (j,j')$ so that  $x^*|_{E_i}\not=0$, then 
 we  can  write $x^*=\sum_{k=1}^\infty x^*_k$ and $x^*=\sum_{k=1}^\infty y^*_k$, with
  $\rg_E(x^*_{k_0})\subset (0,j)$ and $i \in \rg_E(y^*_{k_0})\not\subset (0,j)$ and thus $(x^*_k)$ and $(y^*_k)$ are as elements of
  $Z^*$ different.
\end{rem}
In our  next step we will show that $(Z_j)$ is shrinking in $Z$ and  we first need to recall some notion for families  of finite subsets of $\N$.
  
\noindent{\bf Notation.}
For any set $M$ we denote by 
$[M]$,
$[M]^{<\omega}$ and $[M]^{\omega}$ 
 the subsets,  the finite subsets, and the  infinite subsets of $M$, respectively. 
 If $M=\N$ we introduce the following convention for subsets of $\N$. When we write 
 $A=\{a_1,a_2,\ldots, a_n\}\in [\N]^{<\omega}$ or $A=\{a_1,a_2,a_3\ldots \}\in [\N]^{\omega}$ it is implicitly assumed that the $a_j$ are increasing.

  For $A\in [\N]^{<\omega} $ and $B\in [\N]$ we write $A<B$ if $\max A<\min B$, and introduce the convention that $\emptyset>A$ and $\emptyset <A$ 
  for all $A\in[\N]^{<\omega}$.
   We say 
  that $B$ is an {\em extension of } $A$, and write $A\preceq B$, if $B=A\cup B'$ for some $B'\in [\N]$ with
  $B'>A$.  By $A\prec B$ we mean that $A\not=B$ and $A\preceq B$.
 
 We  identify 
 $[\N]$  in the usual way with the product   $\{0,1\}^\omega$ 
and  consider  on $[\N]$ the product topology  of the discrete topology on  $\{0,1\}$.

 $\cF\ksubset[\N]^{<\omega}$ is called 
 {\em closed under restrictions} if $A\kin\cF$ whenever $A\prec B$ and  $B\kin\cF$,
\emph{hereditary} if $A\kin\cF$ whenever $A\ksubset B$ and $B\kin\cF$,
and $\cF$ is called \emph{compact }if it is compact in the product
topology. Note that a  family  which is closed under restrictions is compact if and only if it is {\em well founded}, \ie if it does  not
contain strictly ascending chains with respect to extensions.
Given $n,\ a_1\kle \dots\kle
a_n,\ b_1\kle \dots\kle b_n$ in $\N$ we say that $\{b_1,\dots,b_n\}$
is \emph{a spread }of $\{a_1,\dots,a_n\}$ if $a_i\kleq b_i$ for $i\keq1,\dots,n$. A family $\cF\ksubset[\N]^{<\omega}$ is called
\emph{spreading }if every spread of every element of $\cF$ is also in
$\cF$.

\begin{lem}\label{L:3.5} $(Z_j)$ is a shrinking FDD of $Z$.
\end{lem}

\begin{proof} Let  $(z_j)$ be  a normalized   block sequence in $Z$ with respect to $(Z_j)$. For any $c\in(0,1)$ we  first show that the set 
$$\cA_c=\big\{ \{m_1, m_2, \ldots, m_l\} :\exists z^*\kin \B^*\,\forall j=1,2,\ldots, l \quad z^*(z_{m_j})\ge c\big\}$$
is compact. Indeed, if our claim were not true, we could find an increasing sequence
$( m_j)\subset \N$ and   $z^*_n\in \B^*$ for each $n\in\N$,  so that 
$z^*_n(z_{m_j})\ge c$ for all $j\le n$. Without loss of generality we can assume that $z^*_n$ converges in $w^*$  to some $z^*$ which 
by  part (4) of Proposition \ref{P:3.3} also lies  in $\B^*$. We write $z^*$ as $z^*=(x^*_k)\subset X^*$,  
and we let $\jb=(j_k)\in \prod_{k=1}^\infty [n_{k},n_{k+1}]$, so that $\rg_E(x^*_k)\subset (j_{k-1}, j_k)$, for $k\in\N$.

 It  follows that $\spa(x^*_k:k\in\N)\subset U_{\jb}$, and  Proposition \ref{P:3.3} yields
that  for $i\in\N$
   $$\Big\|\sum_{k\in \rg_Z(z_{m_i})}  x^*_k\Big\|\ge \sum_{k\in \rg_Z(z_{m_i})}  x^*_k(z_{m_i})= z^*(z_{m_i})\ge c,$$
   which  contradicts the convergence of the series  $\sum_{k=1}^\infty x^*_k$.
   
   We can deduce the rest of the proof  from the following more generally stated  result.
\end{proof}
\begin{lem}\label{L:3.6} Let $V$ be a Banach space having an FDD $(V_j)$ and assume that there is a $1$-norming  subset $B$ of $B_{V^*}$ 
 so that  for some $0<c<1$ and for all normalized block sequences $(v_j)$ in $V$ with respect to $(V_j)$ the set
 $$\cA=\big\{ \{m_1,m_2,\ldots,m_l\} :\exists v^*\kin B\,\forall j=1,2,\ldots, l \quad v^*(v_{m_j})\ge c\big\}$$
 is compact. Then $(V_j)$ is shrinking in $V$.

 Conversely if $(V_j)$ is shrinking then for every $0<c<1$ and every normalized block sequence $(v_j)$  in $V$ with respect to $(V_j)$  the set 
 $$\cB_c=\big\{ \{m_1,m_2,\ldots, m_l\} :\exists v^*\kin B_{V^*}\,\forall j=1,2\ldots, l \quad v^*(v_{m_j})\ge c\big\}$$
is compact.
 
\end{lem}
\begin{proof}   Assume our claim is wrong. Then  by Proposition \ref{P:2.1} we can choose a normalized block sequence  $(v_j)$ in $V$ which is not weakly null. Thus, for some
$\rho\in(0,1)$ it follows that 
\begin{equation}\label{E:3.11}
\Big\|\sum_{j=1}^\infty b_j v_j\Big\|\ge \rho \sum_{j=1}^\infty b_j,\text{ whenever  $(b_i)\kin c_{00}$, with $b_i\ge 0$, for $i\kin\N$}.
\end{equation}
Let $\vp=(1-c)/3$.
Using James' argument \cite{Ja} that $\ell_1$ is not distortable, we can, by passing to further normalized blocks of the $v_j$, assume that 
$\rho>1-\vp$.

 Let $\cA$ be defined as in the statement and note that $\cA$ is hereditary.  We recall the  {\em Schreier space  $X_{\cA}$} defined for $\cA$:
For  $(a_i)\in c_{00}$ we put 
$$\|(a_i)\|_{X_{\cA}} =\sup_{A\in \cA}\sum_{i\in A} |a_i|,$$
and let  $X_{\cA}$ be the completion of $c_{00}$ with respect to $\|\cdot\|_{X_{\cA}}$. We note that 
 the unit vector basis $(e_i)$ is a 1-unconditional basis of $X_{\cA}$ and that 
for $\ab=(a_i)\in X_{\cA}$ the map
$f_{\ab}:\cA\to \R, \quad B\mapsto \sum_{n\in B} a_n$, is a continuous function defined on the countable and compact space $\cA$. 
We compute
$$\|f_{\ab}\|_{C(\cA)}=\sup_{B\in\cA} \Big|\sum_{n\in B} a_n\Big|=\max\Bigg( \sup_{B\in\cA} \sum_{n\in B, a_n\ge 0} a_n,
 \sup_{B\in\cA} \sum_{n\in B, a_n\le 0}(- a_n)\Bigg)\begin{cases} \ge \frac12 \| a\|_\cA\\ \le  \| a\|_\cA,\end{cases}
$$
where the second equality follows form the fact that $\cA$ is hereditary.
Thus, $X_{\cA}$ isomorphically embeds into $C(\cA)$, the space of continuous functions on the countable and compact space $\cA$. But this means that $\ell_1$ cannot embed into 
$X_{\cA}$. In particular, we can  find a finite  sequence  of non negative numbers 
$(a_i)_{i=1}^l$, with $\sum_{i=1}^l a_i=1$ and $\|(a_i)\|_{X_{\cA}}<\vp$.

Since $\cA$ is compact and countable $C(\cF)$ is isomotrically isomorphic to $C[0,\alpha]$ for a countable ordinal $\alpha$, and thus $X_{\cF}$ is $c_0$-saturated.

 Define $v=\sum_{i=1}^l a_i v_i$  and let $v^*\in B$.
 Then $\{ i: v^*(v_i)\ge c\}\in \cA$, and, thus,
 \begin{align*}
 v^*(v)&=\sum_{i=1}^\infty a_iv^*(v_i)\\
             &=\sum_{ v^*(v_i)\ge c}  a_iv^*(v_i)+\sum_{ v^*(v_i)< c}  a_iv^*(v_i) \\
             &\le \sup_{A\in\cA} \sum_{i\in A}  a_i + c\sum_{i=1}^\infty a_i\le \vp +c=  1-2\vp\le\rho-\vp,
 \end{align*}
and thus $\|v\|=\sup_{v^*\in B}  |v^*(v)|\le \rho -\vp$ which contradicts \eqref{E:3.11}, and finishes the proof of our first claim.

The second claim follows from the $w^*$-compactness of $B_{V^*}$  by replacing $\cB^*$ by  $B_{V^*}$  and arguing as in  the first part of  the proof of Lemma \ref{L:3.5}.
  \end{proof}

\begin{lem}\label{L:3.7} If $X$ is reflexive then  $(Z_j)$ is a  boundedly complete  FDD of $Z$.
\end{lem}
\begin{proof} Assume that $(Z_j)$ is not boundedly complete.  Then we can   find a semi normalized block sequence $(z_j)$, 
  say $\frac1C\le \|z_i\|\le 1$, for all $i\in\N$ and some $C\ge 1$, so that
   $\big\|\sum_{j=1}^n z_j\big\| \le 1$, for all $n\kin\N$. Since the set $\B^*$  is $1$-norming $Z$, we can find 
  $z^*_j\in \B^*$ so that $z^*_j(z_j)\ge 1/2C$. 
   For $i\in\N$,  and define
    $y^*_i=P^{Z^*}_{\rg_Z(z_{2i})}(z^*_{2i})$.
    
  Since $(y^*_i)$ is a semi normalized skipped block sequence with respect to $(Z^*_j)$   it follows from Proposition \ref{P:3.3} (4)
   that the sequence $(I^*(y^*_{i}))$ ($I:X\to Z$  as in  Proposition \ref{P:3.3}) is a semi normalized  skipped block sequence in $D^*$ with respect to $(E_j)$, 
   which is isometrically equivalent to $(y^*_i)$. For any sequence $(a_j)\subset[0,1]$ with $\sum_{j=1}^\infty a_j=1$ it follows that 
   $$\Big\|\sum a_iI^*( y^*_{i})\Big\|=\Big\|\sum a_i y^*_{i}\Big\|\ge \Bigg(\sum a_i y^*_{i}\Bigg)\Bigg(\sum_{i=1}^{2n} z_i\Bigg)\ge \frac{1}{2C}.$$
   Thus no convex block  of $(I^*(y^*_{i}))$ converges in norm to 0, which implies that 
   $(I^*(y^*_{i}))$ cannot converge weakly to 0, and contradicts the assumption that $(F_j)$ is a shrinking FMD of $X^*$.
     \end{proof}
   
   From part (5) of Proposition \ref{P:3.3} we also deduce the following criterium for  $(Z_j)$ being an unconditional FDD.
   It will depend on the choice of the FMD $(E_j)$ of $X$. In Section \ref{S:6} we will deduce a {\em coordinate free condition} on $X$ 
  implying that $(Z_j)$ is unconditional, and thereby deduce the  results of \cite{JZ1,JZ2}.

\begin{prop}\label{P:3.8} Assume that every skipped block basis in $X^*$ with respect to $(F_j)$ is $C$-suppression unconditional. Then $(Z_j)$ is 
$C$- suppression unconditional in $Z$.
\end{prop} 
\begin{proof} 
For each  finite $A\ksubset \N$, we will apply  (5) of Proposition \ref{P:3.3}   to the projection $T\keq P^Z_A: Z\to Z$,
and let $T_k: Z\to Z$, for $k\kin\N$, be defined by $T_k\keq P^Z_{\{k\}}$, if $k\kin A$, and $T_k\keq0$, otherwise. 
  Since every sequence $\xb^*\keq(x^*_j)\in \B^*$  is skipped with respect to 
$(F_j)$ we have for $x^*=\sum a_k x^*_k\kin \spa (x^*_j:j\kin\N)$ that
$$ \Big\|T_{\xb}\Big(\sum_{k=1}^\infty a_k x^*_k\Big)\Big\|=\Big\|\sum _{k=1}^\infty a_k x^*_k \circ T_k\Big\|
=\Big\|\sum_{k\in A} a_k x^*_k\Big\| = \Big\|P_{A}\Big(\sum_{k=1}^\infty a_k x^*_k\Big)\Big\|\le C\Big\|\sum_{k=1}^\infty a_k x^*_k\Big\|,$$
where $P_A$ is the projection  $\overline{\spa(x^*_i:i\in\N)}\to \spa(x^*_i:i\in A), \quad \sum_{i=1}^\infty a_ix^*_i\mapsto 
\sum_{i\in A} a_ix^*_i$, which by assumption is of norm not larger than $C$.
This proves by  part (5) of Proposition \ref{P:3.3} that $\|P^Z_A\|\le C$.
\end{proof}
Up to now we proved, that our given Banach space $X$ with shrinking FMD $(E_j)$ embeds into the Banach space $Z$ which has a shrinking FDD $(Z_j)$. Moreover $(Z_j)$ is boundedly complete if $X$ is reflexive, and $(Z_j)$ is unconditional if  there is a $C\ge1$ so that  all the skipped blocks in $X^*$ with  respect to the biorthogonal sequence $(F_j)$ are $C$-unconditional. 

We now show how to pass from an FDD with certain properties (shrinking, boundedly complete, and unconditional) to a basis with the same properties.

To do so assume that $V$ is  a Banach space with an  FDD $(V_j)$.  After renorming we can assume that $(V_j)$ is bimonotone in $V$.  We can therefore view the duals $V_j^*$, $j\kin\N$,   to be isometrically 
subspaces of $V^*$.
Moreover, in the case that the FDD $(V_j)$ is unconditional, we also can assume,
after the appropriate renorming, that  it is $1$-unconditional.

Let $(\vp_n)\subset (0,1)$ be a null sequence with
$\sum \vp_n<1/3$, and choose for each $n\in\N$ a finite $\vp_n$-net   $(x^*_{(n,i)}:i=1,2,\ldots,l_n) $  in $B_{V^*_n}$.  
It follows that the set 
$$A=\Big\{ \sum a_n  x^*_{(n,i_n)} : (i_n)\in \prod_{n=1}^\infty \{1,2,\ldots,l_n\},\, 
\Big\| \sum_{n=1}^\infty  a_n x^*_{(n,i_n)}\Big\|\le 1\Big\}$$
is $\frac23$-norming  the space $V$. After passing to the equivalent  norm defined by 
\begin{equation}\label{E:3.12a}
\tn v\tn =\sup_{v^*\in A} \big| v^*(v) \big|\text{ for $v\in V$}\end{equation}
we can assume that $A$ is $1$-norming the space $V$, and also note that $(V_j)$ is still bimonotone, respectively  $1$-unconditional with respect to this new  norm.

We put $\Gamma =\big\{ (n,i): \,n\kin \N, \text{ and } i=1,2,\ldots, l_n \big\}$. 
We denote the unit vector basis of $c_{00}(\Gamma)=\big\{(a_\gamma:\gamma\kin\Gamma)\subset \R\!:\! \#\{\gamma:a_\gamma\not=0\}<\infty\big\}$ by  $(e_\gamma:\gamma\in\Gamma)$ and its coordinate functionals 
by   $(e^*_\gamma:\gamma\in\Gamma)$. 
We define
\begin{align}\label{E:3.12}  B&=
\Bigg\{\sum_{n=1}^\infty a_n e^*_{(n,i_n)}:  
\sum_{n=1}^\infty  a_n x^*_{(n,i_n)}\in A\Bigg\}\\
&=\Bigg\{\sum_{n=1}^\infty a_n e^*_{(n,i_n)}:  
 (i_n)\in \prod_{n=1}^\infty \{1,2,\ldots,l_n\},\, 
\Big\| \sum_{n=1}^\infty  a_n x^*_{(n,i_n)}\Big\|\le 1\Bigg\}.\notag
  \end{align}
Then we define on $c_{00}(\Gamma)$ the norm
\begin{equation}\label{E:3.13} 
\|x\|=\sup_{w^*\in B} w^*(x),  \text{ for $x\in c_{00}(\Gamma)$}.
\end{equation}
Let $W$ be the completion of $c_{00}(\Gamma)$ with respect to $\|\cdot\|$.

\begin{thm}\label{T:3.9} Let $V$ and $W$ be as introduced above.
 Then $(e_\gamma:\gamma\in\Gamma)$ is a basis of $W$, where 
the set $\Gamma$ is lexicographically ordered.  The map
$$J:V\to W, \quad \sum v_n\mapsto \sum_{n=1}^\infty \sum_{i=}^{l_n} x^*_{(n,i)}(v_n)  e_{(n,i)}.$$ 
is an isometric embedding of $V$ into $W$ (using the norm in \eqref{E:3.12a}), and
\begin{enumerate}
\item[a)] $(e_\gamma:\gamma\in\Gamma)$ is $1$ unconditional, if $(V_j)$ is $1$-unconditional,
\item[b)] $(e_\gamma:\gamma\in\Gamma)$  is shrinking, if $(V_j)$ is shrinking,and 
\item[c)]  $(e_\gamma:\gamma\in\Gamma)$  is shrinking and boundedly complete if $V$ is reflexive.
\end{enumerate}
\end{thm}
\begin{rem} The construction of $W,$ appears already in \cite[Theorem 1.g.5]{LT}, where it was shown that
$(e_\gamma:\gamma\kin\Gamma)$ is unconditional, if $(V_j)$ is unconditional. In \cite{LT}
 the space $W$ is defined by its unit ball, not by its norming set $B$.
 It was already mentioned in \cite{JZ2} that this construction leads to a shrinking basis in the  case that $(V_j)$ is a shrinking FDD. Nevertheless we will, to be 
 self-contained, 
present the complete argument, and  later we will show that the space $W$ has the same Szlenk index as the space $V$, and that, in the case that
$V$ is reflexive,  also $W^*$ and $V^*$ share the same Szlenk index.
\end{rem}

\begin{proof}  First we prove that $(e_\gamma:\gamma\in \Gamma)$, ordered lexicographically, is bimonotone.
 Indeed, denote the lexicographical order on $\Gamma$ by $\preceq$. 
For $\gamma_-=(m, j_-)$ and $\gamma_+=(n, j_+)$ in $\Gamma$, with $m\le n$ and $j_-<j_+$, if $m=n$,  and 
$w^*=\sum_{k=1}^\infty a_k e^*_{(k,i_k)}\in B$ it follows from the bimonotonicity of $(V_j)$ that
$$P^*_{[\gamma_-,\gamma_+]}(w^*):=\sum_{\gamma_-\preceq(k,i_k)\preceq \gamma_+ }^\infty a_k e^*_{(k,i_k)}= 
\left\{\begin{matrix}     \sum_{k=m}^{n}   a_k e^*_{(k,i_k)}                    &      \text{if $j_-\kleq i_m$ and $i_n\kleq j_+$,}\\
                              \sum_{k=m+1}^{n}   a_k e^*_{(k,i_k)}    &  \text{if $j_-\kgr i_m$  and $i_n\kleq j_+$,}\\
                          \sum_{k=m}^{n-1}   a_k e^*_{(k,i_k)}     &\text{if $j_-\kleq i_m$ and $i_n\kgr j_+$,}\\
                           \sum_{k=m+1}^{n-1}   a_k e^*_{(k,i_k)}     &\text{if $j_-\kgr i_m$ and $i_n\kgr j_+$,} \end {matrix} \right.$$
                   and since the set $A$ is closed under projections of the form $P^V_{[i,j]}$ it follows that $P_{[\gamma_-,\gamma_+]}(w^*)\in B$.
                           This yields for  $w=\sum \xi_\gamma e_\gamma \in c_{00}(\Gamma)$ that 
\begin{align*} \Big\|P_{[\gamma_-,\gamma_+]}   \Big(\sum \xi_\gamma e_\gamma    \Big)\Big\|=\Big\| \sum_{\gamma_-\preceq \gamma\preceq\gamma_+} 
 \xi_\gamma e_\gamma\Big\|&=
 \sup_{w^*\in B} P^*_{[\gamma_-,\gamma_+]}(w^*)(w)\le \|w\|.
\end{align*}
For $v=\sum_{n=1}^\infty v_n\in V$, with $v_n\in V_n$, for $n\in\N$, we have  
\begin{align*}
\Big\|J\Big(\sum_{n=1}^\infty v_n\Big)\Big\|&=\sup\Bigg\{\sum_{n=1}^\infty   a_nx^*_{(n,i_n)}(v_n)  :
 (i_n)\in \prod_{n=1}^\infty \{1,2,\ldots,l_n\},\, 
\Big\| \sum_{n=1}^\infty  a_n x^*_{(n,i_n)}\Big\|\le 1\Bigg\}\\
&=\sup\Bigg\{\Big(\sum_{n=1}^\infty   a_nx^*_{(n,i_n)}\Big)\Big(\sum_{n=1}^\infty v_n\Big)  :
  \sum_{n=1}^\infty  a_n x^*_{(n,i_n)}\in A \Bigg\}\\
&=\Big\|\sum_{n\in\N}^\infty  v_n\Big\|, \\
\end{align*} 
and thus $J$ is an isometric embedding of $V$ into $W$.
Assume that $(V_j)$ is $1$-unconditional.  In order to show that $(e_\gamma:\gamma\in\Gamma)$ is $1$-unconditional we observe   
for $(\xi_\gamma:\gamma\in\Gamma)\in c_{00}(\Gamma)$ and $(\sigma_{\gamma}:\gamma\in \Gamma)\kin\{-1,1\}^\Gamma$ that
\begin{align*}
\Big\|\!\sum_{\gamma\in\Gamma} \sigma_\gamma \xi_\gamma e_\gamma\Big\|
 &= \sup\!\Bigg\{\sum_{n=1}^\infty a_n  \sigma_{(n,i_n)}  \xi_{(n,i_n)} \!:\!
 (i_n)\kin \prod_{n=1}^\infty \{1,2,\ldots,l_n\},\, 
\Big\| \sum_{n=1}^\infty  a_n x^*_{(n,i_n)}\Big\|\kleq 1\Bigg\} \\
 &= \sup\!\Bigg\{\!\sum_{n=1}^\infty   a_n\xi_{(n,i_n)} \!:\!
 (i_n)\kin \prod_{n=1}^\infty \{1,2,\ldots,l_n\},
\Big\| \sum_{n=1}^\infty  a_n x^*_{(n,i_n)}\Big\|\kleq 1\!\Bigg\}
\keq\Big\|\sum_{\gamma\in\Gamma}  \xi_\gamma e_\gamma\Big\|,
\end{align*}
where the second equality follows from the equivalence
$$\Big\| \sum_{n=1}^\infty  a_n x^*_{(n,i_n)}\Big\|\le 1\iff \Big\| \sum_{n=1}^\infty  a_n \sigma_{(n,i_n)}x^*_{(n,i_n)}\Big\|\le 1.$$

For $n\kin\N$  put $W_n\keq\spa(e_{(n,j)}\!:\!j\keq1,2,\ldots,l_n)$, and note that $(W_n)$ is an FDD of $W$.
Let $(w^*_m)\subset B$ be a normalized block with respect to $(W^*_n)$. For $m\in\N$ we write $w^*_m=\sum_{j=k_{m-1}+1}^{k_m} a_j x^*_{(j,i_j)}$, for some 
sequences $(i_j)\in\prod_{j=1}^\infty \{1,2\ldots,l_j\}$, $k_1<k_2<\ldots$ in $\N$,  and $(a_j)\subset \R$. For $m\in\N$ define $v^*_m=w^*_m|_V\in V^*$.
On the one hand it follows for any $(b_m)\subset c_{00}$, that 
$\big\|\sum_{m=1}^\infty b_m w^*_m\|\ge  
\big\|\sum_{m=1}^\infty b_m v^*_m\big\|$. On the other hand, if 
 $\big\|\sum_{m=1}^\infty b_m v^*_m\big\|=\big\| \sum_{m=1}^\infty \sum_{j=k_{m-1}+1}^{k_m} a_j x^*_{(j,i_j)}\big\|=1$,
 then $\sum_{m=1}^\infty b_m v^*_m\in A$, and thus $\sum_{m=1}^\infty b_m w^*_m\in B$, which, by definition of the norm on $W$ means that  
 $\big\|\sum_{m=1}^\infty b_m w^*_m\big\|\le 1$. We thus proved that the sequences $(w^*_m)$  and  $(w^*_m|_V)$
are  isometrically equivalent.

Assume now that $(V_j)$ is shrinking.
To show that $(e_\gamma:\gamma\kin\Gamma)$  is shrinking, it will be enough to 
  show that $(W_n)$  is a  shrinking FDD of 
$W$. 
Assume that this were not true, and that by Lemma \ref{L:3.6} 
 there is  a $0<c<1$, a  normalized block  $(w_j)$ in $W$ with respect to $(W_j)$, 
an  increasing sequence $(m_j)$ and 
for each $n\in\N$ an element $w^*(n)\in B$, so that $w^*(n)(w_{m_j})\ge c$, for all $j=1,2,\ldots, n$. After passing to a subsequence we can assume 
$w^*=w^*-\lim_{n\to\infty} w^*(n)$ exists.  Put 
$w^*_k=P^{W^*}_{(\max\rg_W(w_{m_{k-1}}), \max\rg_W(w_{m_{k}})]}(w^*)$, for $k\in\N$.
It follows that $\|w^*_k\|\ge |w^*_k(w_{m_k})|\ge c$, and that 
$\big\|\sum_{k=1}^n w^*_k\big\|\keq\big\|P^{W^*}_{[1, \max\rg_W(w_{m_{k\|}})]}(w^*)\big\|\kleq 1.$
Thus, the previously observed equivalence between $(w^*_k)$ and $(w^*_k|_V)$ yields 
that $\|w^*_k|_V\|\kgeq c$,  for $k\kin\N$, and  
$\big\|\sum_{k=1}^n w^*_k|_V\big\|\kleq 1$, for $n\kin\N$,
which  contradicts  the assumption that $(V_j)$ is shrinking in  $V$ and thus that $(V^*_j)$ is boundedly complete in $V^*$.

Finally assume that $V$ is reflexive. Again we only need to show that $(W_j)$ is boundedly complete.
Assume that $(W_n)$ is not boundedly complete and that we can find a  normalized block $(w_j)$ in $W$ with respect to 
$(W_j)$  so that  $C=\sup_{n\in\N}\big\|\sum_{j=1}^n w_j\big\|< \infty$. For each $n\in\N$ 
 we choose $w^*_n\in B$ so that $w^*_n(w_n)\ge 1/2$. Since 
 $(W_n)$ is
 bimonotone, we can assume 
that also $(w^*_n)$ is a block sequence in $W^*$ with respect to $(W^*_j)$, and therefore, we deduce from the isometric equivalence 
between  $(w^*_j)$ and  $(w^*_j|_V)$,
that 
$$\Big\|\sum_{j=1}^n w^*_j|_V\Big\| \ge  \Big\|\sum_{j=1}^n w^*_j\Big\|  \ge\frac1C \Big(\sum_{j=1}^n w^*_j\Big)\Big(\sum_{j=1}^n w_j\Big)\ge \frac{n}{2C}.$$
which is a contradiction to the assumption that $V$ is reflexive.
\end{proof}

\section{Ordinal Indices for Trees}\label{S:4}

The  aim of this section is to introduce certain ordinal indices of 
trees, and prove some results which will  later be needed to compute the Szlenk indices of the spaces $Z$ and $W$, as defined in Section \ref{S:3}. 
We   first follow the exposition of  \cite{OSZ} and  recall some of the notation introduced there. We begin with defining a  general class of ordinal indices of trees
on arbitrary sets.  

Let $M$ be an arbitrary set. We set $M^{<\omega}\keq
\bigcup_{n=0}^\infty M^n$, the set of all finite sequences in $M$,
which includes the sequence of length zero which is  $\emptyset$. For
$x\kin M$ we shall write $x$ instead of $(x)$, \ie we identify $M$
with sequences of length~$1$ in $M$. A \emph{tree on $M$ }is a 
non-empty 
subset $\cA$ of $M^{<\omega}$ closed under taking initial segments: if
$(x_1,\dots,x_n)\kin \cA$ and $0\kleq m\kleq n$, then
$(x_1,\dots,x_m)\kin\cA$. A tree $\cA$ on $M$ is \emph{hereditary }if
every subsequence of every member of $\cA$ is also in $\cA$.

Given $\xb\keq (x_1,\dots,x_m)$ and $\yb\keq(y_1,\dots,y_n)$ in
$M^{<\omega}$, we write $(\xb,\yb)$ for the concatenation of $\xb$
and $\yb$:
$$
(\xb,\yb)=(x_1,\dots,x_m,y_1,\dots,y_n) .
$$
Given $\cA\ksubset M^{<\omega}$ and $\xb\kin M^{<\omega}$, we let
$$
\cA(\xb)= \{\yb\kin M^{<\omega}:\,(\xb,\yb)\kin\cA\} .
$$
Note that if $\cA$ is a tree on $M$, then so is $\cA(\xb)$
(unless it is empty). Moreover,
if $\cA$ is hereditary, then so is $\cA(\xb)$ and $\cA(\xb)\ksubset
\cA$.

Let $M^\omega$ denote the set of all (infinite) sequences in $M$. Fix a set of $M$-valued sequences 
$S\ksubset M^\omega$. For a tree $\cA$ on $M$ \emph{the $S$-derivative
  $\cA_S'$ of $\cA$ }consists of all finite sequences $\xb\kin
M^{<\omega}$ for which there is a sequence $(y_i)_{i=1}^\infty\kin S$
with $(\xb,y_i)\kin\cA$ for all $i\kin\N$.  Note that $\cA'_S\ksubset\cA$, but that  in general $\cA'_S$ does not need to be a tree. Nevertheless if we assume that $\cA$ is hereditary, then 
 $\cA'_S$ is also a hereditary  tree
(unless it is empty). We then define
higher order derivatives $\cA\sder{\alpha}$ for ordinals
$\alpha\kle\omega_1$ by recursion as follows.
\begin{align*}
  \cA\sder{0} &= \cA,\,\,
  \cA\sder{\alpha+1} = \big( \cA\sder{\alpha}\big)'_S,\,\text{for   }\alpha\kle\omega_1\text{ and }
  \cA\sder{\lambda} = \bigcap _{\alpha <\lambda}
  \cA\sder{\alpha}\quad \text{for a limit ordinal  }\lambda\kle\omega_1.
\end{align*}

It is clear that $\cA\sder{\alpha}\ksupset \cA\sder{\beta}$, whenever
$\alpha\kleq\beta$, and if  $\cA$  is a hereditary  tree it follows that $\cA\sder{\alpha}$ is also  a hereditary  tree
(or the empty set). An easy induction also shows that
\begin{equation}\label{E:4.0}
\big(\cA(\xb)\big)\sder{\alpha} = \big(\cA\sder{\alpha}\big)
(\xb)\qquad \text{for all }\xb\kin M^{<\omega},\ \alpha\kle\omega_1\ .
\end{equation}
We now define \emph{the $S$-index $\si(\cA)$ of $\cA$ }by 
$$
\si(\cA) = \min \{\alpha\kle\omega_1:\,\cA\sder{\alpha}\keq \emptyset\}
$$
if there exists $\alpha\kle\omega_1$ with $\cA\sder{\alpha}\keq
\emptyset$, and $\si(\cA)\keq\omega_1$ otherwise.

We note  for  $\xb\kin M^{<\omega}$, an hereditary tree $\cA\subset [M]^{\omega}$  and  $\alpha\kle\omega_1$ that
\begin{align}\label{E:4.0a}
\si\big(\cA(\xb)\big)\ge\alpha+ 1\iff  \emptyset \in \cA^{(\alpha)}(\xb) \iff \xb\in  \cA^{(\alpha)}\text{ and }\\
\label{E:4.0b}
\si\big(\cA(\xb)\big)\ge\alpha+2 \iff  \exists (y_j)\kin S\, \forall j\kin\N \quad \si\big(\cA(\xb,y_j )\big)\ge\alpha+1.
\end{align}

\begin{rem}
  If $\lambda$ is a limit ordinal and
  $\cA\sder{\alpha}\kneq\emptyset$ for all $\alpha\kle\lambda$, then
  in particular $\emptyset\kin\cA\sder{\alpha}$ for all
  $\alpha\kle\lambda$, and hence
  $\cA\sder{\lambda}=\bigcap_{\alpha<\lambda} \cA\sder{\alpha}\kneq\emptyset$. This shows that $\si(\cA)$ is
  always a successor ordinal.
\end{rem}

\begin{exs}\label{Ex:4.1}

 A tree $\cF\ksubset [\N]^{<\omega}$ (see the notation introduced in Section \ref{S:3})  can be thought of
  as a tree on $\N$: a set $F\keq\{m_1,\dots,m_k\}\kin[\N]^{<\omega}$, with $m_1<m_2<\ldots< m_k$, is
  identified with  the increasing sequence $(m_1,\dots,m_k)\kin\N^{<\omega}$.   Let $S$ be the set of all strictly increasing sequences in $\N$. In
  this case the $S$-index of a   hereditary tree
  $\cF\ksubset[\N]^{<\omega}$ is nothing else but the {\em Cantor-Bendixson
  index} which we denote by  $\CB(\cF)$  of $\cF$. For  the derivative, or more generally, the $\alpha$-derivative of  $\cF\ksubset [\N]^{<\omega}$, with respect to $S$,
  we will  use $\cF'$ and $\cF^{(\alpha)}$, instead of    $\cF'_S$ and $\cF^{(\alpha)}_S$.
  Recall that the {\em Cantor-Bendixson derivative  of $\cF$} for a hereditary tree  $\cF\subset[\N]^{<\omega} $, 
  is
  $$\cF'= \cF'_{[\N]^{\omega}}=\left\{\{a_1,a_2,\ldots,a_l\} : \begin{matrix}\exists \{n_j:j\in\N\}\ksubset\big[\{ a_l+1,a_l+2,\ldots\}\big] ^{\omega}\\
                       \{a_1,a_2,\ldots,a_l, n_j\}\in \cF, \text{ for all $j\in\N$}\end{matrix}\right\},$$
  Note that if $\cF$ is  compact,  then $\cF'$ is compact, and  $\cF'\subset \cF$.
  As already noted in Section \ref{S:3}, if $\cF\subset[\N]^{<\omega}$ is closed under restrictions, then $\cF$ is compact if and only if it  is {\em well founded}, \ie does not contain an 
  infinite chain, and thus every $A\in\cF$ can be extended to a maximal element in $\cF$. We denote the maximal elements  of $\cF$ by $\MAX(\cF)$.
 Since $[\N]^{<\omega}$ is a Polish space, we deduce that the Cantor-Bendixson index $\CB(\cF)$, of  a hereditary tree $\cF\subset [\N]^{<\omega}$ is countable if and only if 
 $\cF$ is compact.

 If $M$ is an arbitrary set and $S\keq M^\omega$ (which includes the constant sequences), then the
  $S$-index of a hereditary tree $\cA$ on $M$ is what is usually called \emph{the
  order of $\cA$ }(or \emph{the height of $\cA$}) denoted by
  o$(\cA)$. Note that in this case the $S$-derivative of $\cA$
  consists of all  non maximal elements of $\cA$.
  The function o$(\cdot)$ is the largest index: for any
  $S\ksubset X^\omega$ we have o$(\cA)\geq \si(\cA)$.

\end{exs}

We say that $S\ksubset X^\omega$ \emph{contains diagonals }if every
subsequence of every member of $S$ also belongs to $S$ and for
every sequence $(x_n)$ in $S$ with $x_n\keq
(x_{(n,i)})_{i=1}^\infty$ there exist $i_1\kle i_2\kle\dots$ in $\N$
such that $(x_{(n,i_n)})_{n=1}^\infty$ belongs to~$S$.

One way to compute  ordinal indices of hereditary trees  on general sets, is to find order isomorphisms  between them and the 
{\em Schreier Sets} $\cS_\alpha$  and the  {\em Fine Schreier Sets} $\cF_\alpha$,  for $\alpha<\omega_1$, which we want to recall now.
 We first fix for every limit
ordinal $\alpha<\omega_1$ a sequence $(\lambda(\alpha,n))_{n\in\N}$ of ordinals with $1\kleq
\lambda(\alpha, n)\nearrow\alpha$.  We want to make sure that $\cF_{\omega^\alpha}=\cS_\alpha$, 
for all $\alpha<\omega_1$, and therefore need to make a very specific choice 
 for $(\lambda(\alpha,n))_{n\in\N}$ which we define by transfinite induction for all limit ordinals $\alpha$.
If $\alpha=\omega$ we put  $\lambda(\alpha,n)=n$ and assuming that 
  $(\lambda(\gamma,n))_{n\in\N}$ has been defined for all limit ordinals $\gamma<\alpha$, we first write 
  $\alpha$ in its {\em Cantor Normal Form} which for a limit ordinal has the (uniquely defined) form
  $$\alpha=\omega^{\xi_l} k_l+\omega^{\xi_{l-1}} k_{l-1} +\ldots+ \omega^{\xi_1} k_1$$
  with $l\in\N$,  $\xi_l>\xi_{l-1}>\ldots> \xi_1\ge 1$ and $k_1,k_2,\ldots, k_l\in\N$ and put 
  $$\lambda(\alpha,n)
=\begin{cases}  \omega^{\xi_l}+  \lambda(\omega^{\xi_{l-1}} k_{l-1}\kplus\omega^{ \xi_{l-2}} k_{l-2}\kplus \ldots\kplus \omega^{\xi_1} k_1,n) &\text{ if $l\kge 2$,}\\
                              \omega^{\zeta} n  &\text{ if $l\keq1$ and $\xi_l=\zeta+1$,}\\
                      \omega^{\lambda(\xi_l,n)} &\!\!\!\!\!\!\!\!\!\!\!\!\!\!\!\!\!\!\!\!\!\text{ if $l\keq1$, $\xi_l$ is limit ordinal, and $\xi_l\kle\omega^{\xi_l}$,}\\ 
                                      \omega^{\beta_n}                               &\!\!\!\!\!\!\!\!\!\!\!\!\!\!\!\!\!\!\!\!\!\text{ if $l\keq1$, $\xi_l$ is limit ordinal, and $\xi_l\keq\omega^{\xi_l}$,}
                      \end{cases} $$
where in the fourth case we choose an arbitrary but fixed sequence $(\beta_n)\subset [0,\xi_l)$ which increases to $\xi_l$.

We define the {\em fine Schreier families}
$(\cF_\alpha)_{\alpha <\omega_1}$ by recursion:
\begin{align*}
  \cF_0 &= \{ \emptyset \}, \quad
  \cF_{\alpha+1} = \big\{ \{n\}\cup A:\,n\kin\N,\ A\kin\cF_\alpha,\big\} \cup \{\emptyset\}\\
  \cF_{\alpha} &= \big\{ A\kin[\N]^{<\omega}:\,\exists n\kleq\min A,
  A\kin\cF_{\lambda(\alpha,n)} \big\},\text{ if $\alpha$ is a limit ordinal.}
\end{align*}
 An easy induction shows that $\cF_\alpha$ is a hereditary, compact and spreading family for all
$\alpha\kle\omega_1$. Moreover, $(\cF_\alpha)_{\alpha<\omega_1}$ is an
``almost'' increasing chain:
\begin{equation}
  \label{E:4.1}
  \forall \alpha\kleq\beta\kle\omega_1\quad \exists n\kin\N\quad \forall
  F\kin\cF_\alpha\qquad\text{if $n\kleq \min F$, then
  $F\kin\cF_\beta$}\ .
\end{equation}
This can also be proved by an easy induction on $\beta$. We note also that
for $A\kin\cF_\alpha\ksetminus\MAX(\cF_\alpha)$ we have
$A\cup\{n\}\kin\cF_\alpha$ for all $n\kge\max A$.
Using transfinite induction it follows for the Cantor Bendixson index   of the fine Schreier families that
$\CB(\cF_\alpha)=\alpha+1$ for all $\alpha<\omega_1$.
The fact that $\cF_\alpha$ is spreading  implies moreover that 
\begin{equation}\label{E:4.2}
\CB\big(\cF_\alpha\cap[N]^{<\omega}\big)=\alpha+1, \text{ for all $\alpha<\omega_1$ and all $N\in[\N]^\omega$.}
\end{equation}
We define the {\em  Schreier family of order $\alpha$}  by
$\cS_\alpha\keq \cF_{\omega^\alpha}$ for all
$\alpha\kle\omega_1$. This is not exactly how the Schreier families
are usually defined, but thanks to our special choice of the sequence $(\lambda(\alpha,n))_{n\in\N}$ for limit ordinals 
both definitions coincide as noted in the following proposition. We will also 
put $\cS_{\alpha,n}=\cF_{\omega^\alpha\cdot n}$, 
 for $\alpha<\omega$ and $n\in\N$.

\begin{prop}\label{P:4.3} Let $\alpha<\omega_1$ and $n\kin\N$.
\begin{align}\label{E:4.3.1}
\cS_{\alpha,n}&=\Big\{\bigcup_{j=1}^n E_j: E_j\in\cS_\alpha,\,  j=1,2\ldots,n, \text{ and } E_1<E_2<\ldots<E_n\Big\} \\
\label{E:4.3.2}\cS_\alpha&=\Big\{ \bigcup_{j=1}^n E_j: n\kleq\min(E_1),\, E_1\kle E_2\kle\ldots\kle E_n,\, E_j\kin\cS_{\beta}, j=1,2\ldots, \Big\} \text{ if  $\alpha\leq\beta\kplus1$}\\
 \label{E:4.3.3}  \cS_\alpha&=\{  E: \exists \, k\le \min(E), \text{ with $E\kin\cS_{\lambda(\alpha,k)}$} \} \text{ if $\alpha$ is a limit ordinal.}
\end{align}
\end{prop}
\begin{proof}[Sketch] We first prove the following claim by transfinite induction for all $\alpha<\omega_1$.

\noindent{\bf Claim}. Assume the Cantor normal form of $\alpha$ is
 $$\alpha=\omega^{\xi_l} k_l+\omega^{\xi_{l-1}} k_{l-1} +\ldots+ \omega^{\xi_1} k_1$$
  with $l\in\N$,  $\xi_l>\xi_{l-1}>\ldots>\xi_1\ge 0$ and $k_l,k_{l-1},\ldots, k_1\kin\N$. Then for all ordinals $\beta$ of the form  
$$\beta=\omega^{\xi_{l+m}} k_{l+m}+\omega^{\xi_{l+m-1}} k_{l+m-1} +\ldots+ \omega^{\xi_{l+1}} k_{l+1},$$
 with $m\in\N$,  $\xi_{l+m}>\xi_{l+m-1}>\ldots> \xi_{l+1}\ge\xi_l$ and $k_{l+m},k_{l+m-1},\ldots,k_{l+1}\kin\N$, it follows that 
 $$\cF_{\beta+\alpha}=\cF_\alpha\sqcup_< \cF_\beta := \big\{ E\cup F: E\in \cF_\alpha, F\in \cF_\beta, E<F\big\}.$$ 
Using the claim we can prove \eqref{E:4.3.1} by induction for all $n\in\N$. Then 
 \eqref{E:4.3.2} and  \eqref{E:4.3.3} follow by transfinite induction, where in the induction step 
  \eqref{E:4.3.2}  follows from   \eqref{E:4.3.1}, the definition of $\cS_\alpha$ 
 and the choice of $(\lambda(\omega^\alpha,n):n\in\N)$ if 
$\alpha$ is a successor ordinal, and   \eqref{E:4.3.3}  follows from the definition of $\cS_\alpha$ and the choice of  $(\lambda(\omega^\alpha,n):n\in\N)$ if
 $\alpha$ is a limit ordinal.
\end{proof} 

For our next observation we need the following notation.
Given a   family $\cF\ksubset[\N]^{<\omega}$ on $\N$, and a family
$(x_F)_{F\in\cF\setminus\{\emptyset\}}$ in  a set $M$, indexed by $\cF$,
$(\xb_F)_{F\in\cF}$ denotes the set of corresponding branches, i.e.
$\xb_\emptyset=\emptyset$ and 
for $F=\{m_1,m_2,\ldots , m_l\}\in\cF\setminus\{\emptyset\}$ we let 
$$\xb_F= \big( x_{\{m_1\}}, x_{\{m_1,m_2\}}, \dots,x_{\{m_1,m_2,\ldots,m_l\}}).$$


\begin{prop}\label{P:4.4}\cite[Proposition 5]{OSZ}
  Let $M$ be a set and assume that  $S\ksubset M^\omega$   contains diagonals. Then for a  hereditary tree $\cA$ on   $M$ and $\alpha<\omega_1$ the following are equivalent.
\begin{itemize}
  \item[(i)]
    $\alpha<\si(\cA)$.
  \item[(ii)]
    There is a family $\big(x_F\big)_{F\in\cF_\alpha\setminus\{\emptyset\}} \ksubset  M$ such that
      $\big(\xb_F\big)_{F\in\cF_\alpha} \ksubset  \cA$ and 
    for all $F\kin \cF_\alpha\ksetminus \MAX(\cF_\alpha)$ the sequence
    $\big(x_{F\cup\{n\}}\big) _{n>\max F}$ is in~$S$.
\end{itemize}
\end{prop}
\begin{rem}\label{R:4.5}  Let $\alpha<\omega_1$ and $\cA\subset M^{<\omega}$ be a hereditary tree. Assume that the family 
$\big(x_F\big)_{F\in\cF_\alpha\setminus\{\emptyset\}} \ksubset  M$ satisfies the conditions in (ii) of Proposition \ref{P:4.4}.
Then the map 
$$\pi:\cF_\alpha\to \cA, \quad \pi(\emptyset)=\emptyset, \,\,  \pi(F)=\xb_F \text{ if $F\in\cF_\alpha\setminus\{\emptyset\}$}$$
is an order isomorphism from $\cF_\alpha$ to $\cA$, such that
that $\pi(F\cup\{n\})=\big(\pi(F), \{x_{F\cup\{n\}}\}\big)$, if $n>\max(F)$, and 
$(x_{F\cup\{n\}}:n>\max(F))\in S$ whenever  $F\in \cF_\alpha\setminus\MAX(\cF_\alpha)$.

In the case of $M=\N$ and $S=[\N]^\omega$ (see Examples \ref{Ex:4.1})  we deduce therefore that
if $\cA\subset[\N]^{<\omega}$ is hereditary and  compact, then $\CB(\cA)> \alpha$ if and only if
there is  an order isomorphism
$\pi:\cF_{\alpha} \to \cA$,  so that 
for all  $A\in\cF_\alpha\setminus \MAX(\cF_\alpha)$ and $n>\max(A)$ 
it follows that  $\pi(A\cup \{n\})=\pi(A)\cup\{s_n\}$, where $(s_n)$ is an increasing sequence
in $\{s\in \N: s>\max\pi(A)\}$.
\end{rem}
\begin{ex}\label{Ex:4.6}\emph {The weak index. }Let $X$ be a separable Banach space. Let
  $S$ be the set of all weakly null sequences in $S_X$, the unit
  sphere of $X$. We call the $S$-index of a hereditary tree $\cF$ on $S_X$
  \emph{the weak index of $\cF$ }and we shall denote it by
  $\wi(\cF)$. We shall use the term \emph{weak derivative }instead of
  $S$-derivative and use the notation $\cF'_{\mathrm{w}}$ and
  $\cF^{(\alpha)}_{\mathrm{w}}$.
  When the dual space $X^*$ is separable, the weak topology on
  the unit ball $B_X$ of $X$ is metrizable. Hence in this case the set
  $S$ contains diagonals and
  Proposition~\ref{P:4.4} applies.
\end{ex}

We  now recall two important results on Schreier families. The first one can be found in \cite{Ga} and is an application of Ramsey's Theorem.
\begin{lem}\label{L:4.9} {\rm \cite[Theorem 1.1]{Ga} }Assume that $\cF,\cG\subset [\N]^{<\omega}$ are two hereditary families and $M\in[\N]^{\omega}$. Then there is
 an $N\in[M]^{\omega}$ so that 
 $\cF\cap[N]^{<\omega}\subset \cG\text{ or }\cG\cap[N]^{<\omega}\subset \cF.$
In particular, if  $\CB\big(\cF\cap[M]^{<\omega}\big)<\CB\big(\cG\cap[M]^{<\omega}\big)$, for all $M\in [\N]^\omega$, then the second 
alternative cannot happen, and thus, for all $M\kin [\N]^\omega$there is an $N\kin[M]^{\omega}$ so that 
$\cF\cap[N]^{<\omega}\subset \cG$.
\end{lem}
In order to state the next result due to Argyros and Gasparis \cite{AG} we will need  further notation and the following observation,
which can be easily shown by transfinite induction.

\begin{lem}\label{L:4.10} Let $\alpha<\omega_1$ then
\begin{enumerate}
\item  $A\in\MAX(\cS_{\alpha+1})$ if and only if  $A=\bigcup_{j=1}^n A_j$, with $n=\min(A_1)$ and $A_1<A_2<\ldots<A_n$ are in $\MAX(\cS_\alpha)$.
 In this case the  sets $A_j$, $j=1,2\ldots,n$ are unique.
 \item If $\alpha$ is a limit ordinal then   $A\in\MAX(\cS_{\alpha})$ if and only if there exists an $n\le \min(A)$ so that 
 $A\in\MAX(\cS_{\lambda(\alpha,n)})$ and for all $k\in\N$, $k>\max(A)$ 
 it follows that $A\cup\{k\} \not\in\bigcup_{j=1}^{\min(A)} \cS_{\lambda(\alpha,j)}$.
\end{enumerate}
\end{lem}
For each $\alpha<\omega_1$ and each 
$A\kin \MAX(\cS_\alpha)$ we will introduce  a probability measure $\P_{(\alpha,A)}$  on $\N$  whose support is  $A$.
If $\alpha=0$ then $\cS_0=\cF_1$ consists of singletons and for $A=\{n\}\in \cS_0$ we put $\P_{(0,\{n\})}=\delta_n$, the Dirac measure in $n$. 
Assume for all $\gamma<\alpha$ and all $A\kin\MAX(\cS_\gamma)$ we already have introduced
$\P_{(\gamma, A)}$ which we write as 
$\P_{(\gamma, A)}=\sum_{a\in A} p_{(\gamma, A)}(a) \delta_a$, with $p_{(\gamma, A)}>0$ for all $a\kin A$. If 
$\alpha=\gamma+1$ for some $\gamma<\omega_1$ and if $A\kin \MAX(\cS_\alpha)$ we  write 
by Lemma \ref{L:4.10} (1)  $A$ in a unique way  as  $A=\bigcup_{j=1}^n A_j$, with $n=\min A$ and $A_1<A_2<\ldots< A_n$ are maximal 
in $\cS_{\gamma}$. We then define
$$\P_{(\alpha, A)} =\frac1n \sum_{j=1}^n  \P_{(\gamma, A_j)}=\frac1n \sum_{j=1}^n \sum_{a\in A_j} p_{(\gamma, A_j)}(a) \delta_a,$$
and thus $$p_{(\alpha, A)}(a)=\frac1n p_{(\gamma, A_j)} (a)\text{ for $j=1,2\ldots,n$ and $a\in A_j$}.$$
If $\alpha$ is a limit ordinal and $A\in\MAX( \cS_{\alpha})$ then 
$$m=\min\{ n\le \min(A): A\in \cS_{\lambda(\alpha, n)}\}$$ 
exists and by Lemma \ref{L:4.10} (2) we have that $A\in\MAX(\cS_{\lambda(\alpha,m)})$ and can therefore put
$$\P_{(\alpha,A)}=\P_{(\lambda(\alpha,m),A})=\sum_{a\in A} p_{(\lambda(\alpha,m),A)}(a) \delta_a.$$ 
The following result was, with slightly different notation, proved in \cite{AG}.
\begin{lem}\label{L:4.11} {\rm  \cite[Proposition 2.15]{AG}}
For all $\vp>0$, all $\gamma<\alpha$, and all $M\in[\N]^{\omega}$, there is an $N=N(\gamma,\alpha, M,\vp)\in[M]^{\omega}$, so that  $ \P_{(\alpha,B)}(A)<\vp$
for all 
$B\in \MAX(\cS_\alpha\cap[N]^{\omega})$ and  $A\in\cS_\gamma$.
\end{lem}
If $\alpha<\omega_1$ and $A\in\MAX( \cS_\alpha)$ we denote the expectation of  a function $f:A\to\R$ with respect to $\P_{(\alpha,A)}$ by $\E_{(\alpha,A)}(f)$.
We finish this section with  the following Corollary  of Lemma \ref{L:4.11}. It  will be used later to estimate the Szlenk index of  Banach spaces.
\begin{cor}\label{C:4.12} For each  $\alpha<\omega_1$ and  $A\in\MAX( \cS_\alpha)$  let $f_A:A\to [-1,1]$  have the property that 
$\E_{(\alpha,A)} (f)\ge \rho$,  for some fixed number  $\rho\in [-1,1]$.
For  $\delta>0$ and $M\in[\N]^{\omega}$ put
$$\cA_{\delta,M}= \Big\{ A\in \cS_\alpha\cap[M]^{<\omega}:\exists B\kin \MAX(S_\alpha\cap[M]^{<\omega}), A\subset B,  \text{ and } f_B(a)\ge \rho-\delta \text{ for all } a\kin A\Big\}.$$
Then $\CB(\cA_{\delta,M})=\omega^\alpha\kplus1$.  
\end{cor}
\begin{proof} Assume our claim is not true. Then we choose  $\gamma<\alpha$ and  $k\kin\N$ so that $\CB(\cA_{\delta,M})<\omega^{\gamma}k$.
Indeed, if $\alpha$ is a successor ordinal we  choose $\gamma$ to be the predecessor of $\alpha$ and $k\in\N$ large enough and if $\alpha$ is limit ordinal  we  choose 
$\gamma<\alpha$ large enough and $k=1$.
Thus, $\CB(\cA_{\delta,M})<\CB(\cS_{\gamma,k})=\omega^{\gamma}k+1$. By Lemma \ref{L:4.9}  and the fact that 
$\CB(\cS_{\gamma,k}\cap[N]^{<\omega})=\omega^{\gamma}k+1$, for  all $N\in[M]^{\omega}$, we deduce that
there is an $N\in[M]^\omega$ so that  
$\cA_{\delta,N}\subset S_{\gamma,k} $. 

Let $0<\vp<\delta/2k$. We can use Lemma \ref{L:4.11} and assume that,  after possibly replacing  $N$ by an infinite  subset, that  
$\P_{(\alpha,B)} (A)<\vp$ for all $B\in\MAX(\cS_\alpha\cap[N]^{<\omega})$ and all $A\in \cS_\gamma\cap[N]^{<\omega}$. But this implies that for all 
 $B\in\MAX(\cS_\alpha\cap[N]^{<\omega})$ that $\{b: f_B(b)\ge \rho-\delta\}\in S_{\gamma, k}$ and thus
 $$\E_{(\alpha,B)}(f_B)\le \rho-\delta + \P_{(\alpha,B)}(\{ b\in B:  f_B(b)\ge \rho -\delta\})\le \rho-\delta+ k \vp<\rho-\delta/2$$
which contradicts our assumption on the expected value  of   $f_B$.
\end{proof}

\section{The Szlenk index of $Z$ and $W$}\label{S:5}

Let $X$ be our space with  separable dual and let $(E_n)$ be   a shrinking  FMD which together with its biorthogonal sequence$(F_n)$ satisfies the conclusions of Lemma  \ref{L:2.2}.
The main goal of this section is to show that the space $Z$, as constructed in Section \ref{S:3}, has the same Szlenk index as $X$, and that 
also $Z^*$ and $X^*$ share the same Szlenk index if $X$ is reflexive.
 Secondly we will prove that the space $W$ constructed from  a space $V$ with FDD $(V_j)$ beforeTheorem \ref{T:3.9} has the same Szlenk
 index as $V$, and that $W^*$ and $V^*$ have the same Szlenk indices if $V$ is reflexive. We thereby  verified  part  (a) and (b)  of our Main Theorem.
We first recall the definition and basic properties of the Szlenk
index. We then  prove further properties that are relevant
for our purposes, including the statement of Theorem C.
 
Let $K$ be
a non-empty bounded subset of $X^*$. For $\vp\kge 0$ the {\em $\vp$-derivative of }$K$ is 
\begin{align*}
K'_\vp &= \Big\{ x^*\kin X^*: \exists (x^*_i)_{i\in I}\ksubset K\text{ net, } \quad x^*_i\overset{w^*}\to_{i\in I} x^*, \text{ and } \|x^*-x_i\|\ge \vp \Big\}\ ,\\
&=
 \Big\{ x^*\kin X^*: \exists (x^*_n)_{n\in\N}\ksubset K  \quad x^*_n\overset{w^*}\to x^*, \text{ and } \|x^*-x_n\|\ge \vp \Big\}\ .
\end{align*}
The  second equality follows from the assumption that $X$ is separable, which yields that the  $w^*$-topology is metrizable on bounded subsets of $X^*$.
It is easy to see that $K'_\vp$ is  a $w^*$-compact  subset of $\overline{K}^{w^*}\!.$  Moreover, it is clear that if  $K \subset \tilde K\subset X^*$ are bounded, then 
$K'_\vp\subset \tilde K'_\vp$, and that $(rK)'_\vp=r\big( K'_{\vp/r}\big)$ for  $\vp,r>0$.
Next, we define for a  bounded set $K\subset X^*$, $\vp>0$ and an ordinal  $\alpha$ the {\em $(\alpha,\vp)$-derivative of }$K$ recursively  by 
\begin{align*}
  K^{(0)}_\vp &= K, \,\,
  K^{(\alpha+1)}_\vp = \big( K^{(\alpha)}_\vp
  \big)'_\vp\text{ for }\alpha\kle\omega_1,\text{ and }
  K^{(\lambda)}_\vp = \bigcap _{\alpha <\lambda}
  K^{(\alpha)}_\vp\,\,\text{for  limit ordinals  }\lambda\kle\omega_1.
\end{align*}

 It was shown  in \cite{Sz} that our assumption that $X^*$ is separable is equivalent  with the property that  for every bounded $K\subset X^*$ the  {\em $\vp$-Szlenk index of $K$ }, defined by  
$$ \Sz(K,\vp )<\min\{ \alpha<\omega_1:  K^{(\alpha)}_\vp= \emptyset\} ,$$
exists.
We define {\em the Szlenk index of $K\subset X^*$} and the {\em Szlenk index of $X$ by }
$$\Sz(K)=\sup_{\vp>0}  \Sz(K,\vp )\text{ and }\Sz(X)=\Sz(B_{X^*})=\sup_{\vp>0}  \Sz(B_{X^*},\vp ).$$

\begin{rem}\label{R:5.1}
  The original definition of $K'_\vp$ in~\cite{Sz} is slightly different, 
 and might lead to different $\vp$-Szlenk indices. Nevertheless it gives the
   same values of $\Sz(K)$ and $\Sz(X)$.    It is also not hard to see that $\Sz(X)$, and more generally $\Sz(K)$, for bounded $K\subset X^*$, are invariant under renormings of $X$.

  For our purposes it will also be important to recall the result of \cite[Theorems 3.22 and  4.2]{AJO} which states that $\Sz(X)$ is always  of the form $\Sz(X)=\omega^\alpha$, for some $\alpha<\omega_1$.
  
\end{rem}
 
 The following  equivalent characterization of the Szlenk index is a generalization of  \cite[Theorem 4.2]{AJO} where it was proven for  the case $K=B_{X^*}$. Our proof will be different and uses the properties of the FMD $(E_n)$.

 \begin{lem}\label{L:5.2} For a  $w^*$-compact set  $K\subset X^*$ and $0<c<1$ we define
 $$\cF_c(K)=\cF_c(X,(E_j),K)= \left \{(x_1,x_2,\ldots,x_l)\subset S_X: \begin{matrix}(x_j) \text{ is skipped with resp. to $(E_j)$,}\\
         \exists x^*\kin K \, \forall j=1,2,\ldots,l \quad x^*(x_j)\ge c                \end{matrix}\right\}.$$
 If $K\not=\emptyset$, but  $\|x^*\|<c$ for all $x^*\kin K$ and $x\in S_X$,  put $\cF_c(K)=\{\emptyset\}$ and 
 put $\cF_c(\emptyset)=\emptyset$.
  
 Then 
 $$\Sz(K) = \sup_{c>0} \wi\big(\cF_c(K)\big).$$
 \end{lem}
 \begin{rem}\label{R:5.3}  In the case that $K=B_{X^*}$ the set $\cF_c(B_X^*)$ can be rewritten 
 as 
 $$\cF_c(B_X^*)=\left \{(x_1,x_2,\ldots,x_l)\subset S_X: \begin{matrix} (x_j) \text{ is skipped block with respect to $(E_j)$ and }\\
              \forall a_1,a_2,\ldots,a_l\ge 0\quad   \Big\|\sum_{j=1}^l a_j x_j\Big\|\ge  c\sum_{j=1}^l a_j                 \end{matrix}\right\}.$$
 Indeed ``$\subset $'' is clear, while ``$\supset$'' follows  from applying the Hahn Banach  Theorem to  separate
 $0$ from the  convex hull of the set  $\{x_1,x_2\ldots,x_l\}$  for each $ (x_1,x_2,\ldots,x_l)$ in the left hand set.
 \end{rem}
 \begin{proof}[Proof of Lemma \ref{L:5.2}] Without loss of generality we assume that $K\subset B_{X^*}$
  and show first for $0<\eta<c<1$ that
   \begin{equation}\label{E:5.2.1}
 \big(\cF_c(K)\big)'_{{\mathrm w}}\subset \cF_c(K'_{c-\eta}).
 \end{equation}
Let $(x_1,x_2,\ldots,x_l)\kin  \big(\cF_c(K)\big)'_{{\mathrm w}}$, and let 
$(y_k)\ksubset S_X$  be a $w$-null sequence with  $(x_1,x_2,\ldots,x_l,  y_k)\kin \cF_c(K)$, for $k\kin\N$.
For $k\kin\N$  we  choose
a $x^*_k\kin K$ such that $x^*_k(x_i)\kge c$, for  $i=1,2,\ldots,l$, and $x^*_k(y_k)\ge c$.
Without loss of generality we can assume, after passing to subsequences if necessary, that $x^*_k$ converges in $w^*$ to some $x^*\in K$.
We observe that 
$$\limsup_{k\to\infty} \|x^*_k- x^*\|\ge \limsup_{k\to\infty} (x^*_k- x^*)(y_k)=\limsup_{k\to\infty} x^*_k(y_k)\ge c,$$
where in the equality we used the assumption that $(y_k)$ is weakly null.
It follows therefore that $x^*\in K'_{c-\eta}$ and since 
$x^*(x_i)= \lim_{k\to\infty} x^*_k(x_i)\ge c$ for all $i=1,2\ldots,l$ we deduce that
$(x_1,x_2,\ldots,x_l)\in \cF_c(K'_{c-\eta})$, which finishes the verification of \eqref{E:5.2.1}.

Using a straightforward  induction argument  \eqref{E:5.2.1} yields that for all $\alpha\kle\omega_1$ 
$$\big(\cF_c(K)\big)_{{\mathrm w}}^{(\alpha)} \subset \cF_{c}\big(K^{(\alpha)}_{c-\eta}\big).$$
In particular this yields that $K^{(\alpha)}_{c-\eta}\not=\emptyset$ if $\big(\cF_c(K)\big)_{{\mathrm w}}^{(\alpha)} \not=\emptyset$.
Thus we have $\wi(\cF_c(K))\le \Sz(K,c-\eta)$, for $0\kle\eta\kle c$ which  yields 
$\sup_{c>0} \wi(\cF_c(K))\le  \sup_{c>\eta>0}\Sz(K,c-\eta)=\Sz(K).$

In order to show the reverse inequality we show for $c<\frac13$  and $\eta< c$ that 
\begin{equation}\label{E:5.2.2}
\cF_c(K'_{3c})  \subset \big(\cF_{c-\eta}(K)\big)'_{{\mathrm w}}.\end{equation}
Let $(x_1,x_2,\ldots,x_l)\in \cF_c(K'_{3c}) $ and let $x^*\in K_{3c}'$ such that 
$x^*(x_i)\ge c$, for $i=1,2,\ldots,l$. We choose a sequence $(x^*_k)\subset K$ which converges in $w^*$ to $x^*$,
and for which $\|x^*-x^*_k\|\ge 3c$, for all $k\in\N$.   Without loss of generality  we assume
that $x^*_k(x_i)\kge c-\eta$, for all $k\kin\N$ and $i\keq1,2,\ldots,l$.

Since $(E_n)$ is a shrinking FMD, and thus $(x^*-x^*_k)\in \overline{ \spa(F_j:j\in\N)}$, for $k\kin\N$, we can,  after passing to a subsequence, assume that there is  a
{\em doubly-skipped } block sequence $(z^*_k)$ with respect to $(F_n)$ in $S_{X^*}$ (meaning
$\max \rg_E(z^*_k)<\min \rg_E(z^*_{k+1})-2 $, for $k\in\N$)  so that 
$\lim_{k\to\infty}\|z^*_k- (x_k^*- x^*)\|=0$. Since $(E_n)$ and $(F_n)$ satisfy   property (3) of  Lemma \ref{L:2.2}  we can, possibly by passing to  subsequences,  find a block 
$(z_k)$  (more precisely: $\rg_E(z_k)\subset [\min \rg_E(z^*_k) -1, \max \rg_E(z^*_k) +1]$) with respect to $(E_n)$  in $S_X$ so that $z^*_k(z_k)\ge c\frac{3}{2.5}$, for all $k \in\N$.
Since $(E_n)$ is shrinking $(z_k)$ is weakly null, and after passing to subsequences again, if necessary, we can assume that 
$$x^*_k(z_k)=   z^*_k(z_k)  +  ( x^*_k-x^*  -     z^*_k)(z_k)+x^*(z_k)\ge c-\eta \text{ for all $k\kin\N$}.$$
 
 It follows that $(x_1,x_2,\ldots,x_l,z_k)\in \cF_{c-\eta}(K)$  for all large enough   $k\kin\N$ and thus  it follows that $(x_1,x_2,\ldots,x_l)\in \big(\cF_{c-\eta}(K)\big)'_{{\mathrm w}}$, since $(z_k)$ is weakly null, and thus  yields  \eqref{E:5.2.2}.
 
 Again by  transfinite induction we deduce from  \eqref{E:5.2.2} that for all $\alpha<\omega_1$ (recall the notation 
 of  $\cF^{(\alpha)}_{{\mathrm w}}$ introduced in Example \ref{Ex:4.6} for the derivative with respect to  weak null sequences for trees $\cF$ on $S_X$)
 $$\cF_c(K^{(\alpha)}_{3c})  \subset \big(\cF_{c-\eta}(K)\big)^{(\alpha)}_{{\mathrm w}}.$$
 
 This implies in particular that if $K^{(\alpha)}_{3c}$ is not empty then $\big( \cF_{c-\eta}(K)\big)^{(\alpha)}_{{\mathrm w}}$ is not empty. 
Thus 
  $\Sz(K)= \sup_{c>0} \Sz(K,c) \le  \sup_{c>0} \wi\big(\cF_c(K)\big),$
 which  finishes our proof.
 \end{proof} 
In our next step   we prove  Theorem C using   Corollary \ref{C:4.12}. 

\begin{proof}[Proof of  Theorem C] Since $\Sz(X)$ is always of the form $\omega^\alpha$ for some $\alpha<\omega_1$, and since $\Sz(X)\ge \Sz(K)$,
we have $\Sz(X)\ge  \min \{ \omega^\alpha: \omega^{\alpha}\ge \Sz(K)\}$.

In order to show the reverse inequality, we first assume without loss of generality that $K$ is $1$-norming $X$, because  otherwise we could 
pass to  the equivalent norm defined by 
$$\tn x\tn=\sup_{x^*\in K}  |x^*(x) | \text{ for $x\in X$}.$$
Since $\Sz(K)=\Sz(\overline K^{w^*})$ we also assume that $K$ is $w^*$-compact.
Let $\alpha<\omega_1$ be such that $\Sz(X) = \omega^\alpha$, and assume that our claim is not true and that there is a $\beta<\alpha$, so that    
$\Sz(K)\le \omega^\beta$.
By Lemma \ref{L:5.2}
$\omega^\alpha=\Sz(X)=\sup_{c>0} \CB\big(\cF_c(B_{X^*})\big)$, where
 $$\cF_c(B_{X^*})=\cF_c(X,\!(E_j),\!B_{X^*})= \left \{(x_1,x_2,\ldots,x_l)\ksubset S_X: \begin{matrix}(x_j) \text{is skipped with resp. to $(E_j)$,}\\
         \exists x^*\kin B_{X^*} \, x^*(x_j)\kge c ,\, j\keq1,2\ldots l             \end{matrix}\right\}.$$
Thus, there is  $c\kin(0,1)$ with $\wi\big(\cF_c(B_{X^*})\big)> \omega^{\beta}$. Since $\cF_c(B_{X^*})$ is hereditary we can apply  Proposition  \ref{P:4.4} and
Remark \ref{R:4.5}, and  choose
a family $(x_F)_{F\in \cS_\beta\setminus\{\emptyset\}}\subset S_X$, so that  for every $F=\{m_1,m_2,\ldots,m_l\}\in \cS_\beta\setminus \{\emptyset\}$,
$\xb_F= \big( x_{\{m_1\}}, x_{\{m_1,m_2\}}, \dots,x_{\{m_1,m_2,\ldots,m_l\}}\big)\in \cF_c(B_{X^*})$, and so that for every 
$F\in\cS_\beta\setminus \MAX(\cS_\beta)$, $(x_{F\cup\{n\}}: n>\max(F))$ is a weak null sequence (recall that $\cS_\beta=\cF_{\omega^{\beta}}$). We now  want to apply Corollary
\ref{C:4.12}. 
For every $B=\{n_1,n_2,\ldots,n_l\}\in  \MAX(S_\beta)$, we have that 
$  \Big\|      \sum_{i=1}^l  p_{\beta,B}(n_i)  x_{\{n_1,n_2, \ldots,n_i\}})\Big\|\ge  c$.
We recall that the probability $\P_{\beta,B}$ on $B$, with coefficients $p_{\beta,B}(n_i)$, $i=1,2,\ldots,l$, where introduced in Section \ref{S:4}  before Lemma \ref{L:4.10}.
Since $K$ is $1$-norming and compact  we  choose to every $B=\{n_1,n_2,\ldots,n_l\}\in \MAX(S_\beta)$ an element $x^*_B\in K$, so that 
$$ 
x^*_B\Big( \sum_{i=1}^l  p_{\beta,B}(n_i)  x_{\{n_1,n_2, \ldots,n_i\}} \Big)=  \Big\|      \sum_{i=1}^l  p_{\beta,B}(n_i)  x_{\{n_1,n_2, \ldots,n_i\}}\Big\|\ge c \text{ for all $i=1,2\ldots,l$}.
$$
For every $B= \{n_1,n_2,\ldots,n_l\}\in\MAX(S_\beta) $ we define $f_B: B\to [-1,1]$, $n_i\mapsto x^*_B(x_{\{n_1,n_2, \ldots,n_i\}})$, 
and note that  we can apply Corollary \ref{C:4.12} to the family $(f_B:B\in \MAX(\cS_\beta))$, and obtain
an $M\in[\N]^\omega$ so that  for $\delta=c/2$ we have
$\CB(\cA_{\delta,M})=\omega^\beta +1$ , where 
$$\cA_{\delta,M}= \Big\{ A\in \cS_\beta\cap[M]^{<\infty}:\exists B\kin \MAX(S_\beta\cap[M]^{<\infty}), A\subset B,  \text{ and } f_B(a)\ge \rho-\delta \text{ for all } a\kin A\Big\}. $$
We now     verify (ii) of Proposition \ref{P:4.4} for the hereditary tree  $\cF_{c/2}(K)$, in order to conclude that $\wi(\cF_{c/2}(K))>\omega^\beta$, which would be a contradiction to the assumption that $\Sz(K)=\sup_{r>0}\wi(\cF_{r}(K))\le \omega^\beta$ (for the equality see Lemma \ref{L:5.2}).

By Remark \ref{R:4.5} we find  an order isomorphism $\pi: S_\beta=\cF_{\omega^\beta} \to \cA_{\delta,M}$,
so that     $\pi(A\cup\{n\})\setminus \pi(A)=\{\max\pi(A\cup\{n\})\}$
for all $A\in S_\beta\setminus \MAX(S_\beta)$, and all $n<\max(A)$.

Then define for $F=\{m_1,m_2,\ldots,m_l\}\in \cS_\beta\setminus\{\emptyset\}$, $z_F= x_{\pi(F)} \in S_X$.
Since $\pi(F) \in  \cA_{\delta,M}$, there is a maximal $B$ in $\cS_\beta\cap [M]^{<\omega}$ containing $\pi(B)$,
so that 
$x^*_B(z_{\{m_1,m_2, \ldots,m_i\}})\ge c/2$ for all $i=1,2,\ldots,l$. It follows therefore that $\zb_F=(z_{\{m_1\}},z_{\{m_1,m_2\}},\ldots,z_{\{m_1,m_2,\ldots,m_l\}})\in \cF_{c/2}(K)$ for all $F\in \cS_\beta\setminus\{\emptyset\}$. Secondly,
it follows  for any non maximal
$F\in \cS_\beta$, that 
$(z_{F\cup\{n\}}: n>\max(F))= (x_{\pi(F\cup\{n\})}:n>\max(F))= (x_{\pi(F)\cup\{\max(\pi(F\cup\{n\}))\}}:n>\max F)$ is weakly null. This verifies that 
 $(\zb_F:F\setminus\{\emptyset\})$ satisfies  the conditions in  (ii) of Proposition \ref{P:4.4} and finishes the proof. 
\end{proof}

\begin{rem}\label{R:5.5}   Theorem C states that if $\Sz(X)=\omega^\alpha$  then  for any  $\beta<\alpha$ and set $K\subset B_X^*$, 
which norms $X$, it follows
that $\Sz(K)>\omega^\beta$.
This is the optimal estimate we have for the Szlenk index of a norming set $K$. Indeed, if $X=C\big[0,\omega^{\omega^\alpha}\big]$ then it follows by 
 \cite[Th\'eor\`em, p.91]{Sa} that $\Sz(X)=\omega^{\alpha+1}$.  The set  $K=\big\{\delta_\gamma: \gamma \in [0,\omega^{\omega^\alpha}]\big\}$  of Dirac measures on $ \big[0,\omega^{\omega^\alpha}\big]$ is  norming $X$, and  $\Sz(K)$ equals to the Cantor Bendixson index of  $[0,\omega^{\omega^\alpha}]$ which is  $\omega^{\alpha}\kplus1$.
\end{rem}

We are now in the position to compute the Szlenk index of the space $Z$, which was constructed in Section \ref{S:3}.
We recall the definition of the sets  $D^*\subset X^*$, $B\subset X^*$, $\D^*\subset Z^*$ $\B^*\subset B_{Z^*}$, 
the spaces $U_{\jb}$, $\jb=(j_k)\in \prod_{k=1}^\infty [n_{k}, n_{k+1}]$, and the embedding $I:X\to Z$
defined in and before Proposition \ref{P:3.3}.

\begin{lem}\label{L:5.6} For $K \subset \B^*$, bounded, and $\eta>0$
$$I^*\big(K '_c\big)\subset \big(I^*(K )\big)'_{c-\eta}.$$
For $\alpha<\omega_1$ it follows  that 
 $$I^*\big(K^{(\alpha)}_c\big)\subset \big(I^*(K)\big)^{(\alpha)}_{c-\eta}.$$
\end{lem}
\begin{proof}

Assume $z^*\in K'_c$ and $\eta>0$. Let $(z^*(n):n\kin\N)\subset  K$, with $z^*(n)\to_{n\to\infty} z^*$ with respect to the $w^*$-topology in $Z^*$, and $\|z^*-z^*(n)\|\kge c$ for all $n\kin\N$. Write $z^*$ and $z^*(n)$ as $z^*=(x_k^*:k\kin\N)\kin \B^*$ and 
 $z^*(n)=(x_k^*(n):k\kin\N)\in \B^*$, 
and 
let $x^*\keq I^*(z^*)\keq w^*-\lim_{n\to\infty} I^*(z^*(n)) $,   and  $x^*(n)\keq I^*(z^*(n))$,  for all  $n\in\N$. After passing to subsequences we can assume that there is a sequence $\jb\keq=(j_k)\subset N$, with
$j_k\in[n_{k}, n_{k+1}]  $ so that $\rg_E(x^*_k(n))\subset (j_{k-1},j_k)$, for all $k\in\N$ and all $n\ge k$, and thus $\rg_E(x^*)\subset (j_{k-1},j_k)$, for all $k\in\N$.
Since $x^*$ and $x^*(n)$ are in $U_{\jb}$ it follows from
 Proposition \ref{P:3.3} part  (4)  that $\|x^*(n)-x^*\|=\|z^*(n)-z^*\|\ge c$.
 Since 
 $x^*=\lim_{n\to\infty} x^*(n)$ and $x^*(n)\in I^*(K)$, for $n\in\N$, it follows that
  $x^*\in \big(I^*(K)\big)'_{c}$, which of  proof the first claim.
  The second claim follows by transfinite induction for all $\alpha<\omega_1$.
\end{proof}
\begin{cor} $\Sz(X)=\Sz(Z)$.
\end{cor}
\begin{proof} We apply the second statement of Lemma \ref{L:5.6} to $K=\B^*$ and deduce from it that
$\Sz(\B^*)\le \Sz(B^*)$, since by Proposition \ref{P:3.3} $I^*(\B^*)=B^*$. If $\alpha$ is such  that $\Sz(X)=\omega^\alpha$, it follows from the fact that $\B^*$ is norming $Z$ and    Theorem C that $\Sz(Z)\le \omega^\alpha$, and thus, since $X\subset Z$, that $\Sz(Z)=\Sz(X)$.
\end{proof}

\begin{lem}\label{L:5.7} If $X$ is reflexive then  $\Sz(X^*)=\Sz(Z^*)$.
\end{lem}
\begin{proof} Since $X$ and $Z$ are  reflexive we can change the roles of $X$ and $X^*$  and of $Z$ and $Z^*$ in Lemma \ref{L:5.2} and deduce that 
$$\Sz(X^*) =\sup_{c>0}\wi\big(  \cF_c(X^*,(F_j), B_X)  \big),$$
where 
 $$\cF_c(X^*,(F_j), B_X)= \left \{( x^*_1,x^*_2,\ldots,x^*_l)\subset S_{X^*}: \begin{matrix}(x^*_j) \text{ is skipped with resp. to $(F_j)$,}\\
         \exists  x \kin B_X \, \forall i=1,2,\ldots,l \quad x^*_i(x)\ge c                \end{matrix}\right\}$$
and
$$\Sz(Z^*) =\sup_{c>0}\wi\big(  \cF_c(Z^*,(Z^*_j),B_Z)  \big),$$
where 
 $$\cF_c(Z^*,(Z^*_j),B_Z)= \left \{( z^*_1,z^*_2,\ldots,z^*_l)\subset S_{Z^*}: \begin{matrix}(z^*_j) \text{ is skipped with resp. to $(Z^*_j)$,}\\
         \exists  z \kin B_Z\, \forall i=1,2\ldots,l \quad z^*_i(z)\ge c                \end{matrix}\right\}.$$
We will abbreviate $\cF_c^{Z^*}= \cF_c(Z^*,(Z^*_j),B_Z)$ and  $\cF^{X^*}_c=\cF_c(X^*,(F_j), B_X$) and show that  for $0<c<1/3$ and for $0<\eta<c/3$
\begin{equation}\label{E:5.7.1} \wi(\cF_c^{Z^*})\le \wi(\cF_{c/3 -\eta}^{X^*})\end{equation} 
which, using Lemma \ref{L:5.2}, yields  the statement  of our lemma.
We first prove the following 

\noindent{\bf Claim 1.} If  $(z^*_j)_{j=1}^{l}\in \cF_c^{Z^*}$, then there exists a sequence 
  $(y^*_j)_{j=1}^{l} \subset \D^*$ so that  $\rg_Z(y^*_j)\subset \rg_Z(z^*_j)$, for $j=1,2,\ldots,l$  (and thus  $(y^*_j)_{j=1}^{l}$ is also 
  a skipped sequence  with respect to $Z_i^*$)  and so that 
  the sequence $(I^*(y^*_j))_{j=1}^l   $ is in $\cF_{c/3}^{X^*}$.

To show Claim 1, let $(z^*_j)_{j=1}^{l}\in \cF_c^{Z^*}$ and let $z\in B_Z$ so that  $z^*_j(z)\ge c  $, for all $j=1,2,\ldots, l$.
For $j=1,2,\ldots, l$ let $z_j=P^Z_{\rg_Z(z^*_j)}$. Since  $\B^*$ is $1$-norming $Z$ there are  $\tilde y^*_j\in \B^*$, $j=1,2,\ldots, l$, so that 
$\tilde y^*_j(z_j)= \|z_j\| \ge z^*_j(z) \ge c$.
We define $y^*_j= P^{(Z^*)}_{\rg_Z(z^*_j)}(\tilde y^*_j)/\|P^{(Z^*)}_{\rg_Z(\tilde y^*_j)}( \tilde y^*_j)\|$, for $j=1,2,\ldots, l$. Then $y^*_j\in \B^*$, and since the  projection constant of $(Z_j)$ does not exceed the value 3, it follows  for all $j=1,2,\ldots, $that 
$$y^*_j(z)=y_j^*\big(P^{Z^*}_{\rg_Z(z^*_j)}(z)\big)=y^*_j( z_j)\ge\frac{c}3.$$
Since $(y^*_j)_{j=1}^l$ is a skipped block sequence with respect to $(Z^*_j)$,   which is  in $\D^*$,  Proposition \ref{P:3.3}  (4)  yields that 
the sequence $(x^*_j)_{j=1}^l= (I^*(y^*_j))_{j=1}^l$ is a skipped block sequence in $D^*$ with respect to $(F_j)$ which is isometrically equivalent 
to  $(y^*_j)_{j=1}^l$.
It follows therefore that for all $(a_j)_{j=1}^l\subset[0,\infty)$ we have
$$\Big\|\sum_{j=1}^l  a_j x^*_j\Big\|=\Big\|\sum_{j=1}^l  a_j y^*_j\Big\|\ge \sum_{j=1}^l  a_j y^*_j(z)\ge \frac{c}3\sum_{j=1}^l  a_j .$$
By the Remark  \ref{R:5.3} this yields that $(x^*_j)_{j=1}^l\in\cF_c$ and thus proves our claim.

 \noindent{\bf Claim 2.} We will prove by transfinite induction for all $\alpha\ge 0$ that   if 
 $(z^*_j)_{j=1}^{l}$ is a skipped normalized block sequence with respect to $(Z^*_j)$ and $\wi\big(\cF_c^{Z^*}(z^*_1,z^*_2,\ldots,z^*_l)\big)\ge \alpha+1$, then there is  a
  sequence 
  $(y^*_j)_{j=1}^{l} \in \D^*$,  with $\rg_Z(y^*_j)\subset \rg_Z(z^*_j)$, and so that for
  $(x^*_j)_{j=1}^l= (I^*(y^*_j))_{j=1}^l$ and all $0<\eta<c/3 $ it follows that 
 $\wi\big(\cF_{c/3-\eta}^{X^*}(x^*_1,x^*_2,\ldots,x^*_l)\big)\ge \alpha+1$.
 
 For $\alpha=0$ Claim 2 reduces  to Claim 1 since  by \eqref{E:4.0a}
 $\wi\big(\cF_c^{Z^*}(z^*_1,z^*_2,\ldots,z^*_l)\big)\ge 1$ means that 
$(z^*_j)_{j=1}^{l}\in \cF_c^{Z^*}$. Assume that Claim 2 is true for $\alpha$ and let $0<\eta<c/3$ and 
$(z^*_j)_{j=1}^{l}$ be a skipped normalized block sequence with respect to $(Z^*_j)$  for which $\wi\big(\cF_c^{Z^*}(z^*_1,z^*_2,\ldots,z^*_l)\big)\ge \alpha+2$.
It follows  from \eqref{E:4.0b} that there is a weakly null sequence $(z^*_{l+1}(n))_{n\in\N}\subset  S_{Z^*}$ so that 
   $\wi\big(\cF_c^{Z^*}(z^*_1,z^*_2,\ldots,z^*_l,z^*_{l+1}(n) )\big)\ge \alpha+1$  for all $n\in\N$.
   By our induction hypothesis  (for $\eta/3$ instead of $\eta$)
   we can find for $n\in\N$   a sequence $(y^*_1(n),y^*_2(n),\ldots,y^*_l,y^*_{l+1}(n) )$ in $\B^*$,  so that 
    $\rg_Z(y^*_j(n))\ksubset \rg_Z(z^*_j)$, for $j\keq1,2,\ldots,l$,  and      $\rg_Z(y^*_{l+1}(n))\ksubset \rg_Z(z^*_{l+1}(n) )$, 
    and so that  for $\big(x^*_j(n)\big)_{j=1}^{l+1}=\big(I^*(y^*_j(n))\big)_{j=1}^{l+1}$ it follows that 
        $\wi\big(\cF_{(c-\eta)/3}^X(x^*_1(n), x^*_2(n),\ldots,   x^*_{l+1}(n)) \big)\ge \alpha\kplus1$,  for all $\eta\in(0,c/3)$.

    After passing to subsequences we can assume that
   $x^*_j=\lim_{n\in\N}  x^*_j(n)$ exists (in norm, because the ranges are bounded) 
  for $j=1,\ldots,l$. 
Using this convergence,  
we can choose  $n_0\kin\N$, large enough, so that  for all $n\ge n_0$,  for all sequences   $(x^*_{l+2}, x^*_{l+3},\ldots,x^*_L)\kin 
 \cF_{(c-\eta)/3}^X(x^*_1(n), x^*_2(n),\ldots,x^*_{l+1}(n))$,  and for all  numbers $a_j\ge 0$, $j=1,2,\ldots, L$,  and  we have 
 \begin{align*}
 \Big\|  \sum_{j=1}^l a_j x^*_j &+ a_{l+1}x^*_{l+1}(n) + \sum_{j=l+2}^L a_j x^*_j\Big\|\\
  &\ge \Big\|  \sum_{j=1}^l a_j x^*_j(n) + a_{l+1}x^*_{l+1}(n) + \sum_{j=l+2}^L a_j x^*_j\Big\|-\frac{\eta}2\sum_{j=1}^{L+2} a_i\\
 & \ge \Big(\frac{c-\eta}3 -\frac\eta2\Big)\sum_{j=1}^{L+2} a_i\ge\Big( \frac{c}3 -\eta\Big) \sum_{j=1}^{L+2} a_i, \end{align*}
which  proves that
 $$ \cF_{(c-\eta)/3}^X(x^*_1(n), x^*_2(n),\ldots,   x^*_{l+1}(n))\subset \cF_{(c/3)-\eta}^X(x^*_1, x^*_2,\ldots , x^*_l,  x^*_{l+1}(n)),$$
and thus
  $$\wi\big( \cF_{(c/3)-\eta}^X(x^*_1, x^*_2,\ldots, x^*_l,  x^*_{l+1}(n))\big)\ge 
 \wi\big( \cF_{(c-\eta)/3}^X)(x^*_1(n), x^*_2(n),\ldots,   x^*_{l+1}(n))\big)\ge \alpha+1,$$
 and therefore $x^*_{l+1}(n)\kin \big( \cF_{(c/3)-\eta}^X(x^*_1, x^*_2,\ldots, x^*_l)\big)^{(\alpha)}$, for $n\in\N$.
Since $(x^*_{l+1}(n))_{n=1}^\infty$ is weakly null (being a bounded block sequence in a reflexive space), it follows that
  $\emptyset \kin\big( \cF_{(c/3)-\eta}^X(x^*_1, x^*_2,\ldots, x^*_l)\big)^{(\alpha+1)}$, and thus
 $\wi\big( \cF_{(c/3)-\eta}^X(x^*_1, x^*_2,\ldots, x^*_l)\big)\!\ge\! \alpha\!+\!2$, which finishes the induction step in the case of successor ordinals.
 
 If $\alpha$ is a limit ordinal and $0<\eta<c/3$  and if 
 $(z^*_j)_{j=1}^{l}$ is a skipped normalized block sequence with respect to $(Z^*_j)$ with  $\wi\big(\cF_c^{Z^*}(z^*_1,z^*_2,\ldots,z^*_l)\big)\ge \alpha+1$
we proceed as follows. For every $\beta<\alpha$ we find by our induction hypothesis 
a
  sequence 
  $(y^*_j(\beta))_{j=1}^{l} \in \D^*$, which satisfies the conclusion of Claim 1  and so that  for
   $(x^*_j(\beta))_{j=1}^l= (I^*(y^*_j(\beta)))_{j=1}^l$ it follows that 
 $\wi\big(\cF_{(c-\eta)/3}^{X^*}(x^*_1(\beta),x^*_2,(\beta)\ldots,x^*_l(\beta))\big)\ge \beta$.
 We can assume that 
 $x^*_j= \lim_{n\to\infty} x^*_j(\beta_n)$ exists for all $j=1,2\ldots,l$ and  for some sequence $(\beta_n)\subset (0,\alpha)$ which increases to $\alpha$.
 A similar argument as in the successor case shows that  we can assume
 after passing to subsequences that for all $n\in\N$ 
 $$\cF_{(c-\eta)/3}^X(x^*_1(\beta_n), x^*_2(\beta_n),\ldots,   x^*_{l}(\beta_n))\subset \cF_{(c/3)-\eta}^X(x^*_1, x^*_2,\ldots, x^*_l),$$
and thus that 
$$\emptyset\in \bigcap_{n\in\N} \big(  \cF_{(c/3)-\eta}^X(x^*_1, x^*_2,\ldots, x^*_l)\big)^{(\beta_n)}= 
\big(  \cF_{(c/3)-\eta}^X(x^*_1, x^*_2,\ldots, x^*_l)\big)^{(\alpha)},$$
which yields our claim, and finishes the induction claim.

 The inequality \eqref{E:5.7.1} follows from Claim 2 applied to the empty sequence.
\end{proof}
The next result  proves part (a) and (b) of Theorem B, and applied  to the space $V=Z$  finishes the  proof of (a) and  (b) of  the Main Theorem.
\begin{lem}\label{L:5.9}
Let $V$ be a Banach space with shrinking FDD $(V_j)$, and let $W$ be the space with shrinking basis containing $V$, which was constructed before 
 Theorem \ref{T:3.9}.
 
 Then $\Sz(W)=\Sz(V)$ and, if $V$ is reflexive, then $\Sz(W^*) =\Sz(V^*)$.
\end{lem}
\begin{proof} Let  $\Gamma$ be the set  defined before  Theorem $\ref{T:3.9}$ and $(e_\gamma:\Gamma)$  the basis of $W$ as introduced there. 
For $n\in\N$ let $W_n=\spa(e_{(n,j)}:j\le l_n)$, for $n\in\N$. As  noted in the proof of Theorem \ref{T:3.9}, $(W_n)$ is an FDD of $W$.
Like in the proof of Theorem \ref{T:3.9}
we can assume that the set 
$$A=\Bigg\{ \sum a_n  x^*_{(n,i_n)} : (i_n)\in \prod_{n=1}^\infty \{1,2,\ldots,l_n\},\, 
\Big\| \sum_{n=1}^\infty  a_n x^*_{(n,i_n)}\Big\|\le 1\Bigg\}$$
is $1$ norming the space $V$\!. The set  
 $B=\big\{\sum_{n=1}^\infty a_n e^*_{(n,i_n)}:  
\sum_{n=1}^\infty  a_n x^*_{(n,i_n)}\in A\big\}$
is (by definition) $1$-norming $W$. We also recall the fact, which was obtained in   the proof of Theorem \ref{T:3.9}, that if 
$(w_j)$ is in $B$ and is a block sequence  with respect to $(W_n)$, then the sequence  $(w_j|_V)$ is in $A$ and it is a block sequence with respect to $(V_n)$  which is isometrically equivalent to $(w_j)$.

The proof that $\Sz(W)= \Sz(V)$ is very similar to the proof that $\Sz(Z)=\Sz(X)$, only a little bit easier since $W$and $V$ have an FDD.
We therefore will only sketch it.  Let $\alpha<\omega_1$ so that $\Sz(V)=\omega^\alpha$. By Theorem C it is enough to show that 
$\Sz(B) \le \omega^\alpha$.  In order to accomplish that   we  first  show that
for any compact $K\subset B$, any $0\kle\eta<  c\kle1$ 
\begin{equation}\label{E:5.9.1}  J^*(K'_c)\subset (J^*(K))'_{c-\eta}\end{equation}
where $J: V\to W$ is the embedding, and thus $J^*:W^*\to V^*$ is the restriction operator. Using the fact that $J^*$ is $w^*$-continuous and the fact  
that $R$ maps block sequences in $B$ into isometrically equivalent block sequences in $A$
(which was shown within the proof of Theorem \ref{T:3.9}), this can be done  the same way we proved Lemma \ref{L:5.6}.
From \eqref{E:5.9.1} we then deduce by transfinite induction for all $\alpha<\omega_1$,  and $0<\eta<c$  that 
$R\big(B^{(\alpha)}_c\big)\subset  \big(J^*(B)\big)^{(\alpha)}_{c-\eta}\subset A^{(\alpha)}_{c-\eta}$, and thus that
$Sz(B)\le Sz(A)\le \omega^\alpha$.

\medskip
Now assume that $V$ is reflexive. The  verification that $\Sz(W^*)=\Sz(V^*)$ is  similar to the proof  of Lemma \ref{L:5.7}, and again
 easier since we are dealing now with FDDs. We will therefore also only sketch it.
As in Lemma \ref{L:5.7} we define for $0<c<1$
 $$\cF_c^{V^*}=\cF_c(V^*,(V^*_j),B_V)= \left \{( v^*_1,v^*_2,\ldots,v^*_l)\subset S_{V^*}: \begin{matrix}(z^*_j) \text{ is skipped with resp. to $(V^*_j)$,}\\
         \exists  v \kin B_V\, \forall i=1,2\ldots,l \quad v^*_i(v)\ge c                \end{matrix}\right\},$$
         and
 $$\cF_c^{W^*}=\cF_c(W^*\!,\!(W^*_j),\!B_Z )= \left \{( w^*_1,w^*_2,\ldots,w^*_l)
 \ksubset S_{W^*}:\! \begin{matrix}(w^*_j) \text{ is skipped with resp.\! to \!$(W^*_j),$}\\
         \exists  w \kin B_W\, \forall i\keq1,2\ldots,l \quad w^*_i(w)\ge c                \end{matrix}\right\}\!,$$
and have to show that for any $0<\eta<c$
\begin{equation}\label{E:5.9.2} 
\wi\big(\cF_c^{W^*}\big)\le \wi\big(\cF_c^{V^*}\big).
\end{equation}
Lemma \ref{L:5.2} will then yield that $\Sz(V^*)=\Sz(W^*)$.

First we prove, as in  Lemma \ref{L:5.7}, the following claim:

\noindent{\bf Claim 1}: If $(w^*_j)_{j=1}^l \in \cF^{W^*}_c$, then there is a sequence $(\tilde w^*_j)_{j=1}^l$ in $B$ so that
$\rg_W(\tilde w^*_j)\subset \rg_W(w^*_j)$, for $j=1,2\ldots,l$,  and so that $(J^*(\tilde w^*_j))_{j=1}^l \in \cF^{V^*}_c$.

To show Claim 1 we will use the  bimonotonicity  of $(V_j)$  and $(W_j)$ in $V$ and $W$, respectively, the fact that $A$ and $B$ are  $1$ norming 
$V$ and $W$, respectively, and the fact that $J^*$ is $w^*$-continuous, and maps  block sequences of $B$ to isometrically equivalent
block sequences in $A$.

Secondly we show, analogously to the proof of Lemma \ref{L:5.7}, by transfinite   induction for all $\alpha<\omega_1$, that 
 if $(w^*_j)_{j=1}^l \in \cF^{W^*}_c$ is a skipped block basis, for which 
 $\wi\big(\cF^{W^*}_c(w^*_1,\ldots,w^*_l)\big)\ge \alpha+1$, then there is 
a sequence $(\tilde w^*_j)_{j=1}^l$ in $B$ for which  
$\rg_W(\tilde w^*_j)\subset \rg_W(w^*_j)$, for $j=1,2\ldots,l$, and for which 
 $\wi\big(\cF^{V^*}_{c-\eta}(J^*(\tilde w^*_1),\ldots,J^*(\tilde w^*_l)\big)\ge \alpha+1$.
 \eqref{E:5.9.2} follows then by applying Claim 2 to the empty sequence. 
\end{proof}
\begin{rem} Since in  Lemma \ref{L:5.9}   $(V_j)$ and $(W_j)$  are FDDs (and not only FMDs) it was actually unnecessary to define  the elements 
of $\cF^{V^*}_c$ and  $\cF^{W^*}_c$ to be skipped block bases, in order to obtain the second part of Lemma \ref{L:5.9}.
Nevertheless in order for the argument to also  hold in general FMDs, it is necessary to use skipped block bases.
\end{rem}
\section{Infinite Asymptotic Games with respect to FMDs}\label{S:6}
 
 In this section we present  {\em Infinite Asymptotic Games}, and show how to use them to deduce  embedding results. 
 They where introduced  for spaces with FDD
 in \cite{OS1,OS2}. The name {\em Infinite Assymptotic Games} was coined by Rosendal \cite{Ro} who generalized them   to  a more general setting.
 In this section we  present another generalization of  {\em Infinite Assymptotic Games}  by defining them with respect to Finite Dimensional  Markushevich 
 Decompositions, and deduce as in the FDD case a combinatorial principle (see Theorem \ref{T:6.11}), 
 which can be used to characterize the property that a  certain Banach embeds    into a space with a certain FDD or basis. One of these results is the intrinsic   characterization of
 subspaces of spaces with an unconditional basis
 by Johnson and Zhang \cite{JZ1,JZ2}.   We will  show  that if our space $X$ has the {\em Unconditional Tree Property}, as defined   in \cite{JZ1, JZ2},
  and if we have started out with an appropriately blocked FMD, then the FDD $(Z_i)$ of the  space $Z$ constructed   in Section 3 is  automatically unconditional.
   This will lead to an alternate   proof  of Johnson's and Zhang's results. Actually, some of the ideas, for example the idea of using FMDs instead of FDDs  can already be found in   their second paper \cite{JZ2}. Nevertheless, since we suspect that Theorem \ref{T:6.11} could lead to other interesting embedding results,
    we would like to present it in a more general form.
Some  of our arguments will be  very similar to the arguments in  \cite{OS1,OS2}. But for the sake of a  better readability  and for being   self-contained we present  the
complete arguments.
   
   We start with a general separable Banach space $X$ and we assume that we have chosen a fixed but, for the moment arbitrary,  $1$-norming FMD $(E_n)$
    and denote its biorthogonal sequence by $(F_j)$. We denote by $\cB_\omega=\cB_\omega(X,E)$, $\cB_f(X,E)$, and $\cB_n(X,E)$, $n\kin\N\cup\{0\}$,  the set of  infinite, or finite sequences, or sequences of length $n$ in
$S_X\cap c_{00}(E_j)$ which are block sequences with respect to $(E_j)$
 (we require now that also the last element of  a sequences $(x_j)_{j=1}^l\in \cB_l$ has finite support). 
   
   We consider on $\cB_\omega$ the product topology of the norm topology
on $\cB_1\equiv S_X\cap c_{00}(E_j)$ and denote the closure for  $\cA\ksubset \cB_\omega$ with respect to that topology by $\overline{\cA}$. Note that $\cA\subset  \cB_\omega$ is open  if and only  if
for $(x_j)\kin \cB_\omega$ 
\begin{align} \label{E:6.3} 
(x_j:j\kin\N)\in \cA&\iff \exists  n\kin\N,\, \delta>0\quad\{ (z_i)\in \cB_\omega :  \|x_i-z_i\|\le \delta, \,\,i=1,2,\ldots,n\}\subset \cA
\intertext{ and $\cA$ is closed  if and only if for every $(x_j)\kin \cB_\omega$ }
\label{E:6.4}
(x_j:j\kin\N)\in \cA&\iff  \forall  n\kin\N, \,\delta\kgr 0\,\exists (z_i)\in\cA\quad \|z_i-x_i\|\le \delta, \text{ for all $i=1,2\ldots,n$.} 
 \end{align}
For $\cA\subset \cB_\omega$ and a sequence $\vpb=(\vp_j)\subset\R^+$ we define the {\em  $\vpb$-fattening of $\cA$} by
\begin{align*}
\cA_\vpb&=\big\{(x_j)\in\cB_\omega: \exists (z_j)\in \cA \quad \|x_j-z_j\|\le \vp_j\big\}.
\end{align*}
For $\cA\subset \cB_\omega$ and  $\xb=(x_1,\ldots,x_n)\in \cB_f$ we define
$$\cA(\xb)=\big\{ \zb\in \cA:  \zb\succ\xb  \big\}.$$
Here we mean, as in Section \ref{S:2}, by $\zb\prec \xb$ that $\zb$ is an extension of $\xb$.

\medskip
For $\cA\subset \cB_\omega$ we now consider the following {\em $\cA$-game} between two players:

Player I chooses $k_1\in\N$, then Player II chooses $x_1\in S_X\cap c_{00}(E_j)$ with $\min \supp_E(x_1)\ge k_1$, 
then again Player  I chooses $k_2\in\N$, $k_2>\max\rg_E(x_1)$ and   Player II chooses $x_2\in  S_X\cap c_{00}(E_j)$ with $\min \supp_E(x_2)\ge k_2$.
This goes on for infinitely many steps and  Player I is declared the winner of that game if the resulting
sequence $(x_j)$ lies in $\cA$. 

Let us precisely formulate what it means that   Player II  has a winning strategy for that  $\cA$-game  and observe  what a winning strategy is.  In order to do so,
 we define the {\em full tree on $\N$ by}
 $\cT=[\N]^{<\omega}=\{A\subset \N, \text{ finite}\}$.  On $\cT$ we consider the  order of extensions  $\succ$ introduced in Section \ref{S:4}.
  A {\em full indexed tree } will be a family indexed by $\cT$, in our cases with values in  a Banach space $X$. For simplicity 
  we will in this section often call a full indexed tree $(x_t:t\in \cT)\subset X$ simply a {\em  tree in $X$} if it can not be confused with the type of  trees  which were  considered in Section \ref{S:4}.
   If $(x_t)_{t\in \cT}$ is a  full indexed tree in $X$ and $t\in \cT$, we call the sequence $(x_{(t,k)})_{k>\max(t)}$ a {\em node of }  $(x_t)_{t\in \cT}$
and  if $k_1<k_2<k_3<\ldots $ we call the sequence $(x_{\{k_1,k_2,\ldots,k_j\}}:j\kin\N)$ {\em a branch of } 
$(x_t)_{t\in \cT}$  (note that $x_\emptyset$ is not part of a branch). The  tree  $(x_t)_{t\in \cT}$ is called {\em normalized } if $(x_t)_{t\in \cT}\subset S_X$, {\em weakly null } if every 
node is weakly null and {\em a block tree with respect to  the FMD $(E_j)$, } if every node is a block sequence with respect to $(E_j)$.
More generally, if $\mathcal U$ is any topology on $X$ (for example $\sigma(X,Y)$ for some $Y\subset X^*$), a {$\mathcal U$-null tree} is a tree for which all nodes  are 
$\mathcal U$-null.

 An {\em  indexed  subtree }of a tree $(x_t:t\in \cT)$ is a family  $(x_t:t\in \cS)$  indexed by a  non empty subset $\cS$ of $\cT$, which is a subtree of $\cT$, \ie which is closed under taking restrictions. We call such an indexed subtree {\em well-founded } if $\cS$ is well founded (see Section \ref{S:3}).  We say that   $(x_t:t\in \cS)$ is
 {\em infinitely branching } if every non maximal $s\in\cS$ has infinitely many direct successors.
  Assume that       $(x_t:t\in \cT)$ is a  full indexed  tree and  that $\cT' \subset \cT$,  is a subtree  which has the property
that for all $t\in \cT'$ the set $\{ n\in\N: \{t,n\}\in \cT'\}$ is infinite. 
  We call then  $(x_t:t\in \cT')$ a  {\em full indexed  subtree} of $(x_t:t\in \cT)$.
 It is easy to see that,  that  there is an order isomorphism between $\cT'$ and $\cT$,  and, using that order isomorphism,   we can reorder   $(x_t:\in \cT')$  into
  $(z_t:t\in \cT)$, having the same branches and nodes as $(x_t:\in \cT')$. In that case we also call $(z_t:t \in \cT)$ a full indexed  subtree of $(x_t:t\in \cT)$.

   \begin{prop}\label{P:6.2a}  Assume that $Y$ is a  subspace of $X^*$ which  separates points of $X$. For example $Y$ could be the closed linear span 
   of the biorthogonal sequence  $(F_j)$.

   Let $(x_t:t\in \cT)\subset S_X$ be a  normalized  $\sigma(X,Y)$-null tree and let  $\vpb=(\vp_t:t\in \cT)\subset (0,1)$. Then there is 
 a full subtree $(z_t:t\in \cT)$ of $(x_t:t\in \cT)$ and a block  tree $(\tilde z_t:t\in \cT)\subset S_X\cap c_{00}(E_j)$ with respect to $(E_j)$ so that 
 $\|z_t-\tilde z_t\|<\vp_t$ for all $t\in \cT$. We say in that case that $(\tilde z_t:t\in \cT)$ is a {\em $\vpb$-perturbation of  $(z_t:t\in \cT)$}.
 
 Moreover, let $\cT$ be linearly ordered  into $t_0,t_1,t_2,\ldots $  consistent  with the partial order $\succ$,
   \ie
  if $m<n$, then $t_n$ and $t_m$ are either incomparable with respect to $\prec$ or  $t_m\prec t_n$.  
 Then $(\tilde z_t:t\in \cT)$ can be chosen so that $(\tilde z_{t_n})$ is a block sequence with respect to $(E_j)$.

  \end{prop}
  \begin{proof}[Proof of Proposition \ref{P:6.2a}] Write $\vp_n=\vp_{t_n}$, and assume w.l.o.g. that $\vp_n<\frac12$ for $n\kin\N$.  
  Choose $\tilde z_\emptyset\in S_X\cap c_{00}(E_j)$
   so that $\|\tilde z_\emptyset -x_\emptyset\|<\vp_0$.  Since the node  $(x_{\{n\}}:n\kin\N)$ is $\sigma(X,Y)$-null, and
   thus $\big(P^E_{[1,\max\supp_E( \tilde z_\emptyset)]} (x_{\{n\}}):n\kin\N\big)$ is norm-null,  we can choose $k_1$ large enough, so that
  $$\|P^E_{[1,\max\supp_E( \tilde z_\emptyset)]} (x_{\{k_1\}})\|<\vp_1/5$$
  and choose $s_1=\{k_1\}$ (as element of $\cT$) and
  $$ z'_{\{1\}}=\frac{ P^E_{(N,\infty)} (x_{\{k_1\}})}{\|P^E_{(N,\infty)} (x_{\{k_1\}})\|} \text{ and  } z_{\{1\}}=x_{\{k_1\}},$$
  where $N=\max\supp_E( \tilde z_\emptyset))$.
  It follows that $\| z'_{\{1\}}-z_{\{1\}}\|<\vp_1$.
  Indeed, 
  \begin{align*}
 \Bigg\| x_{\{k_1\}} &-   \frac{ P^E_{(N,\infty)} (x_{\{k_1\}})}{\|P^E_{(N,\infty)} (x_{\{k_1\}})\|}\Bigg\| 
    \le \big\|x_{\{k_1\}}- P^E_{(N,\infty) }(x_{\{k_1\}})\big\| 
+ \big|\|P^E_{(N,\infty)} (x_{\{k_1\}})\|-1\big|
    <  \vp_1.
  \end{align*}
  Then  we can perturb $z'_{\{1\}}$ to an element $\tilde z_{\{1\}}$ in $S_X\cap c_{00}(E_j)$, with $\min\supp(\tilde z_{\{1\}})\ge N$
  still satisfying  $\| \tilde z_{\{1\}}-z_{\{1\}}\|<\vp_1$.
  
  Now assume that we have found $s_0,s_1,s_2,\ldots,s_{k-1}\in \cT$  and a block sequence $(\tilde z_ {t_0} ,\tilde z _{t_1},\ldots,\tilde z _{t_{k-1}})$ so that 
   the set $\cS_{k-1}=  \{s_0,s_1,s_2,\ldots,s_{k-1}\}$  is close under taking  restrictions,
   the map
  $$\{t_0,t_1,t_2, \ldots, t_{k-1}\}\to   \{s_0,s_1,s_2,\ldots,s_{k-1}\},\quad t_j\mapsto s_j$$
 is an order isomorphism,
  and 
  $\|\tilde z _{t_j} - x_{s_j}\|<\vp_j$, for $j=0,1,\ldots, k-1$.
  
   The element $t_k$ has then a direct predecessor $t_j$, $j<k$, with respect to $\prec$ (not necessarily $t_{k-1}$).
  Since the node $(x_{s_j\cup\{n\}}: n>\max(s_j))$ is $\sigma(X,Y)$-null we can find a large enough $n$ so that
  $\|P^E_{[1, N]} (x_{s_j\cup\{n\}})\|<\vp_{k}/5$ where $N=\max\supp_E(x_{s_{k-1}})$. Then we let 
  $s_k= s_j\cup\{n\}$ and find as before 
  $\tilde z_{t_k}\in S_X\cap c_{00}(E_j)$,
 so that $\|\tilde z_{t_k}- x_{s_j\cup\{n\}}\| <\vp_k$, and note that  the
  set $\cS_k= \{s_0,s_1,s_2,\ldots,s_{k}\}$
   is close under taking  restrictions, and 
   the map
  $$\{t_j:j\kleq k\}\to   \{s_0,s_1,s_2,\ldots,s_k\},\quad t_j\mapsto s_j$$
 is an order isomorphism.
 
 This finishes the recursive construction, and we observe that  $(x_s:s\in \cS)$ with $S=\bigcup_{k\in\N} \cS_k$ is a full subtree 
 and $(\tilde z_{t}: t\in \cT)$ is an $\vpb$-perturbation of that subtree, and, moreover, the sequence
 $(\zt_{t_n}: n\in\N_0)$ is a block sequence in $S_X$. 
    \end{proof}

 We now  give a formal description of  what it   means that Player II has a winning strategy.
\begin{rem}\label{R:6.3}
Assume $\cA\subset \cB_\omega$.
Then 
Player II has a winning strategy in the $\cA$-game if and only if there is a  block tree $(x_t:t\in \cT)$ so that no branch is in $\cA$.

Indeed, for $l\in\N$ and $k_1<k_2< \ldots,k_l$ we  define $x_{\{k_1, k_2,\ldots,k_l\}}$, to be the $k$-th choice of Player II following a winning strategy,
assuming Player I has chosen so far $k_1<k_2<k_3<\ldots<k_l$. This defines   a   tree $(x_t:t\kin \cT)$  in $S_X\cap(\oplus E_j)$, which has the property   that  no  branch 
of $(x_t:t\in \cT)$ is  in $\cA$.
Since $\min \supp_E(x_{\{k_1, k_2,\ldots, k_l\}})\ge k_l$, we can pass to a full subtree of $(x_t:t\in \cT)$ for which all nodes are block sequences.

 Conversely, if there is a  block tree $(x_t:t\kin \cT)$ so that no branch is in $\cA$
  we can first assume, after passing to a full subtree, that $\min \supp_E(x_{\{k_1,k_2,\ldots,k_l\}})\ge k_l$ for all $(k_1,\ldots,k_l)\in \cT$.
  Player II can now use this tree as her strategy: If Player I  has chosen  $k_1<k_2<\ldots<k_l$ so far,
 Player II answers with $x_{\{k_1,\ldots, k_l\}}$.  The result of the game is therefore a branch of $(x_t:t\in \cT)$, which by assumption does not lie in $\cA$.  
\end{rem}

\begin{prop}\label{P:6.4}
Assume $\cA\subset \cB_\omega$ is closed and assume that $(x_t:t\in \cT)$ is a winning strategy for Player II as in Remark \ref{R:6.3}. Then there exists a well founded  and infinitely branching subtree 
$(x_s:s\in \cS)$,  so that for every maximal $s\in  \cS$ 
$$\{ \zb: \zb\in\cB_\omega, \zb \succ \xb_s\}\cap \cA=\emptyset$$
Here we mean, as in Proposisition \ref{P:4.4}, by $\xb_t$ for $t=\{t_1,t_2,\ldots,t_l\}\in \cT$, the finite sequence
$$\xb_t=\big(x_{\{t_1\}}, x_{\{t_1,t_2\}},x_{\{t_1,t_2,t_3\}},\ldots,x_{\{t_1,t_2,\ldots, t_l\}}\big)$$
\end{prop}
\begin{rem}\label{R:6.5}  Proposition \ref{P:6.4} means that if $\cA$ is closed and  Player II has a winning strategy, the outcome of the  game is determined after finitely many  (but possibly at the beginning of the game still  undetermined) steps.  
\end{rem}
\begin{proof}[Proof of Proposition \ref{P:6.4}] Define
$$\cS'=\big\{ s\in \cT:   \zb\in\cB_\omega,\zb\succ \xb_s\}\cap \cA\not=\emptyset\big\}\cup\{\emptyset\},$$
and note that $\cS'$ is  closed under taking restrictions.  Secondly,  it is also   well founded.  Indeed, otherwise there would be an  increasing sequence $(k_j)$ in $\N$ so that
$t_l=\{k_1,k_2,\ldots,k_l\} \in \cS'$,  for each $l\in\N$. But this would mean that
for each  $l$ there is a  block sequence $z^{(l)}\in \cB_\omega$ so that $(\xb_{t_l}, \zb^{(l)})\in\cA$. Since $\cA$ is closed this implies 
that the infinite sequence  $(x_{t_l}:l\kin \N)$ is in $\cA$.  Since $(x_{t_l}:l\kin \N)$ is a branch of $(x_t: t\kin \cT)$ this contradicts the assumption we made for $(x_t: t\kin \cT)$.
Now define 
$$\cS=\{ (s,n) : s\in \cS', n\in\N \text{ with } n>\max(s)\}.$$
Then $\cS$ is also well founded, and  no maximal element $s$ of $\cS$ is in $\cS'$ and thus  for every maximal element $s$ in $\cS$ we have $\{ (\xb_s, \zb): \zb\in\cB_\omega\}\cap \cA=\emptyset$. Moreover,
every element which is not maximal in $\cS$ must be in $\cS'$ and has therefore by definition of $\cS$ infinitely many successors.
\end{proof}
The following result was shown by Martin \cite{Mart} for more general games. In the  case that $\cA\subset \cB_\omega$ is closed it has an easy proof (see also \cite{Mart}) . 
\begin{thm}\label{T:6.5} {\rm\cite[Theorem]{Mart}} If $\cA\subset \cB_\omega$ is Borel then the $\cA$-game is determined, meaning that either Player I or Player II has a winning strategy.
\end{thm}
\begin{rem}
From  \cite{Mart}  it actually follows that it is  enough   that $\cA$ is Borel with respect to the product topology of the discrete topology on $S_X\cap c_{00}(E_j)$, to imply that the $\cA$-game is determined.
\end{rem}

\begin{defin}\label{D:6.7} We  say that  $\cA\subset\cB_\omega$ {\em is closed under taking tails } if for every $(x_j:j\in\N)\in \cA$ and any $n\in \N$ it follows that $(x_{j+n}: j\in\N)$ is in $\cA$. 
\end{defin}

\begin{prop}\label{P:6.6} Assume that $(\cA_m:m\kin\N)$ is  an increasing family of closed subsets of $\cB_\omega$ which are closed under taking tails  and let 
 $\cA=\bigcup_{n=1}^\infty\cA_n$. If Player I  has 
a winning strategy for  the $\cA$-game then there is an $m\in\N$, so that she also has a winning strategy for the $\cA_m$-game.
\end{prop} 
\begin{rem} Let us first  present an intuitive argument for the claim  in Proposition \ref{P:6.6}.  
If Player II has a strategy for each $\cA_m$ game, for all $m\in\N$, she
can  use the following winning  strategy for the  $\bigcup_{m\in\N} \cA_m$-game:
 First she follows her strategy for $\cA_1$ and choses  $x_1,x_2,x_3,\ldots $. Since $\cA_1$ is closed, she will, after say $l_1$ moves,  be in the situation that for all $\zb\in\cB_\omega$, with $\zb\succ (x_1,x_2,\ldots,x_{l_1}) $, it follows that  $\zb\not\in \cA$. Then she switches to the  strategy for $\cA_2$, and after Player I chooses $k_{l_1+1}$ she 
 chooses  the  element  $x_{l_1+1}$ of $S_X\cap c_{00}(E_j)$  which she would have chosen, if $k_{l_1+1}$ had been the first step of Player I in the $\cA_2$-game. She follows her strategy 
  choosing $x_{l_{1}+2},x_{l_{1}+3},\ldots$ until,    after some $l_2$ steps, with $l_2>l_1$, she will again be in the situation 
  that for all $\zb\in\cB_\omega$, with $\zb\succ (x_{l_1+1},x_{l_1+2},\ldots,x_{l_{2}})$, it follows that $\zb\not\in \cA$.  She continuous that way and finally produces a sequences
  $(x_j)\in \cB_\omega$ and $(l_j)\subset \N$ so that
  $(x_{l_m+j}:j\in\N)\not \in \cA_m$ for all $m\in\N$. Since $\cA$ is closed under taking  tails, it follows that the whole sequence $(x_j)$ is not in $\cA$ and, thus, that Player II has won.
  
  Since by Theorem \ref{T:6.5} the games $\cA_m$, $m\kin\N$, and $\cA$ are determined, we deduce therefore that, if Player  I has a winning strategy 
  for the $\cA$ game, and thus player II was not a winning strategy for that game, it follows that there is an $m\in\N$ so that  player II has no winning strategy for the $\cA_m$-game, and thus   player I has a winning strategy for that game.
\end{rem}

\begin{proof}[Proof of Proposition  \ref{P:6.6}]  Since by Theorem \ref{T:6.5} the $\cA$-game  and the $\cA_m$-games, $m\kin\N$,  are determined, we need to show that Player II has a strategy for the $\cA$-game, assuming that she has a strategy for each $\cA_m$-game.
By Proposition  \ref{P:6.4}  there is for each $m\kin\N$ a well founded and infinitely branching tree  $(x^{(m)}_s\!:\!s\kin\cS_m)\subset S_X\cap c_{00}(\oplus E_j)$, $\cS_m\subset \cT$,
so that 
 $\{  \zb: \zb\kin\cB_\omega, \zb\!\succ\!\xb^{(m)}_s\}\cap \cA_m=\emptyset$, for each maximal $s\in  \cS_m$. 
After relabeling we can assume that for each non maximal $s=\{k_1,\ldots,k_l\}$ in $\cS_m$ it follows that $\{k_1,k_2,\ldots,k_l,k\}\in \cS_m$ for all $k>k_l$. 
We define a full tree  $(x_t ; t\kin \cT)$ as follows: If $ t=\emptyset$ we put  $x_\emptyset=x^{(1)}_\emptyset $ (this choice is irrelevant)
For any other  $t=\{k_1,k_2,\ldots,k_l\}\in \cT$, $l\ge 1$  we proceed as follows. We choose
$m\in\N$ and $0=l_0<l_1<l_2<\ldots< l_{m-1}<l_m= l$,  so that
for all $1\le j<m$, $\{k_{l_{j-1}}+1,  k_{l_{j-1}}+2, \ldots,   k_{l_{j}}\}$ is a maximal element  of $\cS_j$ and 
$\{k_{l_{m-1}}\kplus1, k_{l_{m-1}}\kplus2, \ldots, n_{l_m}\}$ is a (not necessary maximal) element of $\cS_m$.
Then we define for that $t$
$$x_t= x^{(m)}_{\{k_{l_{m-1}}\kplus1,  k_{l_{m-1}}\kplus 2, \ldots, k_{l_m}\}}.$$
 It follows that each branch  $(z_j)$ of $(x_t:t\in \cT)$ (\ie $z_j=x_{k_1,k_2,\ldots,k_j\}}$ for some increasing sequence $(k_j)\subset \N$) can be subdivided  into finite sequences 
 $(z_{j}:l_{m-1}<j\le l_m)$, for $m\in\N$, so that  for all $\zb\succ (z_{j}:l_{m-1}<j\le l_m)$ we have 
   $\zb\not\in \cA_m$. In particular, $(z_{j}:k_{m-1}<j)\not\in \cA_m$, for all $m\in\N$, and since
   $\cA_m$ is closed under taking  tails, it follows that $(z_j:j\kin\N)\not\in \cA$. Thus $(x_t:t\in\cT)$ is a winning strategy for Player II.
\end{proof}
Let us list some  examples of sets $\cA\subset \cB_\omega$ which are of interest (see \cite{FOSZ,JZ1, JZ2,OS1,OS2,OSZ})
\begin{exs}\label{Ex:6.10} The following sets in (a), (b)  (c)  $\cA\subset \cB_\omega$ are hereditary under taking tails and closed.
The example in (d) is  Borel.
 
 \begin{enumerate}
 \item[a)]  For $C\ge 1$ let
 $$\cA=\big\{(x_j)\in\cB_\omega: (x_j)\text{ is $C$-unconditional}\big\}$$
\item[b)] For $C\ge 1$ and $1\le p\le \infty$
 \begin{align*}
 \cA&=\big\{(x_j)\in\cB_\omega: (x_j)\text{ is $C$-equivalent to the $\ell_p$-unit vector basis }\big\} \text{ or }\\
 \cA&=\big\{(x_j)\in\cB_\omega: (x_j)\text{ $C$-dominates  the $\ell_p$-unit vector basis }\big\} \text{ or }\\
  \cA&=\big\{(x_j)\in\cB_\omega: (x_j)\text{ is $C$-dominated by  the $\ell_p$-unit vector basis }\big\} 
  \end{align*} 
 \end{enumerate}
 We could replace in the examples of (b) the $\ell_p$ unit vector basis by any other basic sequence $(v_j)$. But in the case
 that $(v_j)$ is not sub symmetric (if for example $(v_j)$ is the unit vector basis of a Tsirelson space) the following choice is 
 more meaningful (cf. \cite{FOSZ, OSZ}). 
\begin{enumerate}
\item[c)] Let $(v_j)$ be a normalized basic sequence and $C\ge 1$
\begin{align*}
 \cA&=\left\{(x_j)\in\cB_\omega:  \begin{matrix}\text{$(x_j)$ is $C$-equivalent to $(v_{m_j})$, where for $j\in\N$}\\
                              \text{$m_j\in[\max\supp_E(x_{j}),\max\supp_E(x_{j})]$}  \end{matrix}\right\} \text{ or }\\
 \cA&=\left\{(x_j)\in\cB_\omega:  \begin{matrix}\text{$(x_j)$ $C$-dominates  $(v_{m_j})$, where for $j\in\N$}\\
                              \text{$m_j\in[\max\supp_E(x_{j}),\max\supp_E(x_{j})]$}  \end{matrix}  \right\} \text{ or }\\
  \cA&=\left\{(x_j)\in\cB_\omega:   \begin{matrix}\text{$(x_j)$ is $C$-dominated by $(v_{m_j})$, where for $j\in\N$}\\
                              \text{$m_j\in[\max\supp_E(x_{j}),\max\supp_E(x_{j})]$}  \end{matrix} \right\} .
  \end{align*} 
  \item[d)]  For the next  example we assume that $\cF\subset [\N]^{<\omega}$ is hereditary, spreading and compact,
   $C \ge 1$ and $(v_j)$ is a normalized and subsymmetric  basic sequence 
   \begin{align*}
 \cA&=\big\{(x_j)\in\cB_\omega: \{ A\in [\N]^{<\omega} : (x_j: j\kin A) \text{ is $C$-equivalent to $(v_j:j\in A)$}\}\in \cF\big\} \\
  \cA&=\big\{(x_j)\in\cB_\omega: \{ A\in [\N]^{<\omega} : (x_j: j\kin A) \text{  $C$-dominates $(v_j:j\in A)$}\}\in \cF\big\} \\
  \cA&=\big\{(x_j)\in\cB_\omega: \{ A\in [\N]^{<\omega} : (x_j: j\kin A) \text{ is  $C$-dominated by $(v_j:j\in A)$}\}\in \cF\big\}
   \end{align*} 
Note that the first set in (d) can be written as
$$\cA=\bigcap_{B\in[\N]^{<\omega}\setminus\cF} \big\{(x_j)\in\cB_\omega: (x_j: j\kin B) \text{ is not $C$-equivalent to $(v_j:j\in B)$}\big\}.$$ \\
This implies easily that $\cA$ is Borel. A similar argument works for the two other sets.
\end{enumerate} 
\end{exs} 

We are now ready to state the main result of this section

\begin{thm}\label{T:6.11} Let $\cA\subset \cB_\omega$. The following are equivalent.
\begin{enumerate} 
 \item[a)] For all decreasing sequences $\vpb=(\vp_n)\ksubset(0,1)$  Player I has a winning strategy for the $\overline{\cA_\vpb}$-game.
 \item[b)]   For all decreasing sequences $\vpb=(\vp_n)\ksubset(0,1)$  every block tree $(x_t:t\in \cT)\ksubset S_X$ has a branch  which lies in $\overline{\cA_\vpb}$. 
 \item[c)] For all decreasing sequences $\vpb=(\vp_n)\ksubset(0,1)$   there is an increasing sequence $(m_j)\subset\N$  so that for the blocking
 $(H_j)$, with $H_j=\spa(E_i:m_{k-1}< i\le m_k)$  (with $m_0=0$), the following holds:
Every normalized   skipped block sequence  $(z_i)\ksubset S_X$  with respect to $(H_j)_{j\ge 2}$ lies in $\overline{\cA_\vpb}$.
 \end{enumerate}
and letting $Y$ be the closed linear span   of  $(F_j)$, then above conditions are equivalent with
 \begin{enumerate} 
  \item[d)] For all decreasing sequences $\vpb=(\vp_n)\ksubset(0,1)$  every  $\sigma(X,Y)$  null tree $(x_t:t\in \cT)\subset S_X$ has a branch  which lies in $\overline{\cA_\vpb}$.

 \end{enumerate}
\end{thm}
\begin{proof}[Proof of Theorem \ref{T:6.11}]
The equivalences of (a), (b) and,  (d),  follow from Proposition  \ref{P:6.2a},   and Remark  \ref{R:6.3}.
It is also clear that (c) implies (a). Indeed, assuming (c), Player I has the following easy strategy: She chooses for given $\vp\keq(\vp_n)\ksubset(0,1)$ the sequence $(m_j)$ as in (d).
Her first move will be $n_1\keq m_1$, and after the $k$-th  step, in which Player II has chosen $x_k$, Player I choose $n_{k+1}\keq m_{N+1}$, where
$N\keq\max\supp_H(x_k)$. Therefore she forces  Player II  to  pick a skipped block sequence with respect to $(H_j)_{j\ge 2}$ which lies in $\overline{\cA_{\vpb}}$.

Now assume (a), our goal is to prove (c).  
 Let $\vpb=(\vp_j) \subset (0,1)$ be given.  We can assume that $(\vp_j)$ decreases.
For $n\kin\N$ put $\vpb{(n)}=(\vp_j{(n)}:j\kin\N)=(\vp_j(1-2^{-n}):j\kin\N)$

 We claim that we can recursively  choose $m_1<m_2<m_3<\ldots $  satisfying the following two properties
   (letting $H_j=\spa(E_i:m_{j-1}<i\le m_j)$, $j=1,2,\ldots, n$):
\begin{align}\label{E:6.11.1}
&\text{For every  skipped  block sequence $(x_i)_{i=1}^l$ in $S_X\cap \spa(H_j:2\kleq j\kleq n-1)$  (with  }\\
&\text{respect to $(H_j)$) and every $x\kin S_X\cap\spa(E_j:j \kgr m_n)$ Player 1 has a winning }\notag\\
&\text{strategy for the $\overline{\cA_{\vpb{(n)}}}(x_1,\ldots,x_l, x)$-game.}  \notag\\
\label{E:6.11.2}
&\text{For every skipped  block  $(x_i)_{i=1}^l$ in $S_X\cap \spa(H_j:2\kleq j\kleq n)$  (with respect to $(H_j)$)  }\\
&\text{Player 1 has a winning strategy for the $\overline{\cA_{\vpb{(n)}}}(x_1,\ldots,x_l)$-game.}  \notag
\end{align}
Since by assumption (d) Player 1 has a winning strategy for  $\overline{\cA_{\vpb(1)}}$  there is an $m_1$, so that for  all $x \in S_X \cap \spa(E_j:j\ge m_1)$ Player 1 has a winning 
strategy for the  $\overline{\cA_{\vpb(1)}}(x)$-game.
Note that for $n=1$, $\emptyset$ is the only skipped  block in $\spa(H_j: 2\le j\le 0)$. Thus,
in that case \eqref{E:6.11.1} simply says that   for any $x\in S_X\cap\spa(E_j: j>m_1))$ Player 1 has a winning strategy for the  $\overline{\cA_{\vpb{(1)}}}(x)$-game,
which follows from our choice of $m_1$ 
and \eqref{E:6.11.2}  means that Player 1 has a winning strategy for the $\overline{\cA_{\vp(1)}}$-game, which follows from our assumption (d).

Now assume that $m_1<m_2<\ldots<m_n$ have been chosen so that conditions \eqref{E:6.11.1} and  \eqref{E:6.11.2} hold.
We first choose a {\em dense enough} finite set  $B$ of skipped block sequences with respect to $(H_j:j=2,3,\ldots,n)$, more precisely, $B$ includes the empty block, and for any 
skipped  block sequence $(x_j)_{j=1}^l$ in $S_X$  with respect to $(H_j:j=2,\ldots,n)$, there is a sequence  $b=(\xt_j)_{j=1}^l\in B$ of the same length $l$,
so that $\supp_E(x_j)=\supp_E(\xt_j)$, for $j\keq1, 2,\ldots,l$, and so that $\|x_j\kminus\xt_j\|<\vp_{n+2} 2^{-n-2}$.  Then we choose,  using   \eqref{E:6.11.2},  to each $b\kin B$ a natural number  $k(b)\kgr m_n$, so that $k(b)$
 could be the first move of a winning strategy  for  Player 1 in the $\overline{\cA_{\vpb{(n)}}}(b)$-game. We let  $m_{n+1}=\max_{b\in B} k(b)$ and have to verify 
  \eqref{E:6.11.1} and \eqref{E:6.11.2} for $n\kplus1$. 
  
  To verify   \eqref{E:6.11.1} for $n\kplus1$ let 
 $(x_j)_{j=1}^l$ be a skipped block in $S_X\cap\spa(H_j:2\kleq j\kleq n)$ with respect to $(H_j)_{j=2}^n$. We first choose 
  $(\xt_j)_{j=1}^l\in B$,  so that $\supp_E(x_j)=\supp_E(\xt_j)$, and so that $\|x_j-\xt_j\|<\vp_{n+2} 2^{-n-2}$, for $j=1,2,\ldots, l$.
  Note that  $ \overline{\cA_{\vpb{(n)}}}(\xt_1,\xt_2,\ldots,\xt_l) \subset \overline{\cA_{\vpb{(n+1)}}}(x_1,x_2,\ldots,x_l)$. Indeed,
  \begin{align*}
\overline{\cA_{\vpb{(n)}}}(\xt_1,\ldots,\xt_l)&=\overline{\big\{ (z_j)\subset S_X: (\xt_1,\xt_2,\ldots,\xt_l,z_1,z_2,\ldots ) \kin \cA_{\vpb{(n)}}\big\}}\\
&\subset \overline{\big\{ (z_j)\subset S_X: (x_1,x_2,\ldots,x_l,z_1,z_2,\ldots ) \kin \cA_{\vpb{(n+1)}}\big\}}=\overline{\cA_{\vpb{(n+1)}}}(x_1,\ldots,x_l).
\end{align*}
  The choice of $m_{n+1}$ therefore yields  that  condition \eqref{E:6.11.1} is satisfied for $n+1$.
  Condition  \eqref{E:6.11.2} for $n+1$  follows now from condition \eqref{E:6.11.1} for $n$ if we note that for any normalized  sequence 
  $(x_i)_{i=1}^l\subset S_X \cap \spa(H_j: 2\le j\le n+1) $,
  which is a skipped block with respect to $(H_j)_{j=2}^{n+1}$
  the sequence $(x_j)_{j=1}^{l-1}$ must be a skipped block sequence in  $\spa(H_j:2\le j\le n-1)$.
  This finishes the inductive choice of $(m_j)$ and $(H_j)$.

 Now assume that $(x_i)$ is a skipped block sequence with respect to $(H_j)$ in $S_X$. Then for every initial segment $(x_j)_{j+1}^l$ Player 1 has a  winning strategy for 
 $\overline{\cA_{\vpb}(x_1,\ldots,x_l)}$ in particular this means that  $\cA_{\vpb}(x_1,\ldots,x_l)$ cannot be empty. Thus, since $\overline{\cA_{\vpb}}$ is closed  it follows that   $(x_j)\kin \overline{\cA_{\vpb}}$.
\end{proof}  
From Theorem \ref{T:6.11} we deduce the  missing part of the Main Theorem, namely the verification  that, under the appropriate assumption, the FDD $(Z_i)$ of $Z$ is uncondtional.

According to \cite{JZ2} we say that $X$ {\em has the $w^*$-Unconditional Tree Property} ($w^*$-UTP) if every $w^*$-null tree in $X^*$ has a branch which is unconditional.

\begin{cor}{ \rm{\cite{JZ2}}}
Assume that $X$ has the $w^*$-UTP and that $X^*$ is separable. Then there is an FMD $(E_j)$ with biorthogonal sequence $(F_j)$ so that the space FDD 
$(Z_j)$ of the space $Z$, as constructed in Section \ref{S:3} is unconditional.
\end{cor}
\begin{proof} Let $(E'_j)$  be any shrinking FMD of $X$ and $(F'_j)$ its biorthogonal sequence and define for $C\ge 1$.
\begin{align*}
\cA_C&=\big\{ (x^*_j)\in \cB_\omega(X^*,F'): (x^*_j)\text{ is $C$-unconditional}\}\text{ and  } \\
   \cA&=\bigcup_{m\in\N} \cA_m=
  \big\{ (x^*_j)\in \cB_\omega(X^*,F'): (x^*_j)\text{ is unconditional}\} .
  \end{align*}
As noted in Examples \ref{Ex:6.10} $\cA_C$ is closed. We also note that for any summable and decreasing  sequence $\vpb=(\vp_j)\subset (0,1)$ we have 
$\cA=\cA_{\vpb}$,
and for  $C\ge 1$ there is a $C'=C'(\vp)$ so that $\cA_C\subset [\cA_C]_{\vpb}\subset \cA_{C'}$.
Using the equivalence (a)$\iff$(d) in Theorem  \ref{T:6.11} and  Proposition \eqref{P:6.6} we deduce that there is a $C\ge 1$ so that Player I has a winning strategy for the 
$\cA_C$-game.  But this implies, maybe after increasing $C$ slightly and  using  the  equivalence (a)$\iff$(c) in Theorem  \ref{T:6.11} that we can block $(F'_j)$ into an MFD  $(F_n)$ so that every skipped block in $S_{X^*}\cap \spa(F:j\ge 2)$  with respect  to $(F_j)$ is $C$-unconditional. After possibly increasing $C$ again 
and after possibly passing to further blocks , we can assume that very skipped block in $S_{X^*}\cap \spa(F:j\ge 1)$  with respect  to $(F_j)$ is $C$-unconditional
and that the conclusions of Lemma \ref{L:2.2}  are satisfied.  Therefore our claim follows  from Proposition \ref{P:3.8}.
\end{proof}
\begin{rem} As proved  in \cite[Theorem 2.12]{JZ1} if $X$ is reflexive the property of having the $w^*$-Unconditional Tree Property is equivalent with 
having the {\em $w$-Unconditional Tree Property} which means that every weakly null tree in $S_X$ (not in $X^*$) has a branch which is unconditional.
\end{rem}

\end{document}